\documentclass[12pt]{amsart}
\usepackage{latexsym}
\usepackage{amsfonts}
\usepackage{amsthm}
\usepackage{tikz}
\usepackage{amsmath}
\usepackage{tkz-euclide}
\usepackage{bm}
\usepackage{bbm}
\usepackage{graphicx}
\usepackage{url}
\newtheorem{thm}{Theorem}
 \newtheorem{cor}[thm]{Corollary}
 \newtheorem{lem}[thm]{Lemma}
 \newtheorem{prop}[thm]{Proposition}
\theoremstyle{definition}
 \newtheorem{defn}[thm]{Definition}
\newtheorem{rem}{Remark}
\newtheorem{ex}{Example}

 \DeclareMathOperator{\conv}{conv}
\title{Links in Edgewise Triangulations via Integer Partitions and their shellings}

\author[D. Joji\'{c}]{Du\v{s}ko Joji\'{c}}

\address[D. Joji\'{c}]{ Faculty of Science,
University of Banja Luka, Bosnia and Herzegovina,
dusko.jojic@pmf.unibl.org}

\author[O. Papaz]{Ognjen Papaz}

\address[O. Papaz]{Faculty of Philosophy,
University of East Sarajevo, Bosnia and Herzegovina,
 ognjen.papaz@ff.ues.rs.ba}

 \keywords{}

\begin{document}

\date{}

\begin{abstract}
We show that the combinatorial types of the links of the vertices
in the edgewise triangulation $T_{k,q}$ of a $(k-1)$-simplex are
encoded by the partitions of $k$. Each of these complexes is
isomorphic to a subcomplex of the barycentric subdivision of the
boundary of a $(k-1)$-simplex, and the containment relations among
them are described by a new poset on the set of partitions of $k$.
We compute the $h$-vectors of these complexes and determine the
number of vertices of $T_{k,q}$ whose links are the same
(correspond to the same partition).

 The combinatorial type of the
link of an $(s-1)$-dimensional face of $T_{k,q}$ corresponds to a
partition $(\lambda_1,\lambda_2,\ldots,\lambda_s)$ of $k$ into $s$
parts, together with additional partitions of each $\lambda_i$. We
also enumerate the combinatorial types of all $m$-dimensional
complexes that arise as the links in edgewise triangulations.

 A
new permutation statistic, \textit{the faithful initial part}, is
introduced and used to describe the star cluster of a facet of
$T_{k,q}$. By examining a specific shelling of this star cluster,
we prove that the $i$-th entry of its $h$-vector
 counts the number of permutations of $[k]$ with exactly $i$ descents,
 taking into account the
 faithful initial part as the multiplicity. Finally, we
 describe a concrete shelling order for $T_{k,q}$, give a
 combinatorial interpretation of its $h$-vector,
and derive an explicit formula for it.
\end{abstract}

 \maketitle \setcounter{section}{0}

\section{Preliminaries}
\subsection{Edgewise subdivision}
The edgewise subdivision of a simplex $\Delta$ with parameter $q$
is a specific triangulation that subdivides each $d$-dimensional
face of $\Delta$ into $q^d$ faces of the same dimension. This
construction first appeared in \cite{Fr} for $q=2$, and an
extension for all $q$ is given in \cite{Edel}. Edgewise
subdivision has numerous applications, including graphs and
hypergraphs coloring \cite{M-V}, discrete geometry \cite{Edel},
group actions on combinatorial structures  \cite{Berg}, or
algebraic combinatorics \cite{Br}. We briefly review the
definitions following the notation from \cite{M-V}. For $k,q\in
\mathbb{N}$, let
$$R_{k,q}=\left\{\mathbf{x}=(x_1,x_2,\ldots,x_{k-1})\in
\mathbb{R}^{k-1}:0\leq x_1\leq x_2\leq \cdots\leq x_{k-1}\leq
q\right\}$$ denote the $(k-1)$-simplex in $\mathbb{R}^{k-1}$. Its
vertices are
$$\mathbf{w}_i=(\underbrace{0,\ldots,0}_{k-i},
\underbrace{q,\ldots,q}_{i-1})\textrm{ for all }i\in[k],\textrm{
and } R_{k,q}=
\conv\left\{\mathbf{w}_1,\mathbf{w}_2,\ldots,\mathbf{w}_k\right\}.$$
\noindent The \textit{edgewise subdivision} of $R_{k,q}$ is the
simplicial complex $T_{k,q}$, whose vertex set is
$W_{k,q}=\{\mathbf{v}\in \mathbb{Z}^{k-1}: 0\leq v_1\leq v_2\leq
\cdots\leq v_{k-1}\leq q\}.$ The maximal cells or \textit{facets}
 of $T_{k,q}$ are determined by a vertex and
a consistent permutation. A permutation $\pi=\pi_1
\pi_2\ldots\pi_{k-1}\in
 \mathbb{S}_{k-1}$ is said to be \textit{consistent}
 with $\mathbf{v}\in W_{k,q-1}$ if, whenever $v_i = v_{i+1}$,
 the index $i$ appears before $i+1$ in $\pi$.
 Specifically, for $\mathbf{v} \in
W_{k,q-1}$ and $\pi \in \mathbb{S}_{k-1}$ consistent with
$\mathbf{v}$, the facet $F(\mathbf{v},\pi)$ of $T_{k,q}$ is
$F(\mathbf{v},\pi)=\conv\left\{\mathbf{v}^{(1)},
\mathbf{v}^{(2)},\ldots,\mathbf{v}^{(k)}\right\}$, where
\begin{equation}\label{E:vertices1}
\mathbf{v}^{(1)}=\mathbf{v},
 \mathbf{v}^{(2)}=\mathbf{v}+e_{\pi_{k-1}},
 \ldots\textrm{, }\mathbf{v}^{(k)}=\mathbf{v}+
 e_{\pi_{k-1}}+\cdots+e_{\pi_{1}}=\mathbf{v}+\mathbbm{1}.
\end{equation}

\noindent Informally, $\mathbf{v}$ is the first vertex of
$F(\mathbf{v},\pi)$, and the remaining $k-1$ vertices are obtained
by incrementing the coordinates of $\mathbf{v}$ one at a time,
 reading $\pi$ backwards.

  The number of facets of $T_{k,q}$ is $q^{k-1}$
   (see Lemma 6.2 in \cite{M-V}), and this result is obtained
geometrically. We describe a bijection between the facets of the
triangulation $T_{k,q}$ and the sequences from
$\mathcal{S}_{k,q}=\{0,1,\ldots,q-1\}^{k-1}$, which provides an
alternative notation for the facets used throughout the paper (a
similar notation appears in \cite{Thom}).
\begin{prop}\label{P:facetsarestrings} The number
of $(k-1) $-dimensional simplices in $T_{k,q}$ is $q^{k-1}$.
\end{prop}
\begin{proof}
 The facet $F=F(\mathbf{v},\pi)$ of $T_{k,q}$
 can be encoded as
$\overline{\mathbf{v}}_F=(v_{\pi_1},v_{\pi_2},\ldots,v_{\pi_{k-1}})
\in \mathcal{S}_{k,q}.$ We prove that the map
$F(\mathbf{v},\pi)\mapsto \overline{\mathbf{v}}_F$ is a bijection
by finding its inverse map. For
$\mathbf{a}=(a_1,a_2,\ldots,a_{k-1})\in \mathcal{S}_{k,q}$, we
choose the unique permutation
 $\alpha=\alpha_1\alpha_2\ldots\alpha_{k-1}\in \mathbb{S}_{k-1}$
that encodes the order of the coordinates of $\mathbf{a}$:
\begin{equation}\label{E:Order}
a_{\alpha_1}\leq a_{\alpha_2}\leq\cdots \leq a_{\alpha_{k-1}},
\textrm{ and } a_{\alpha_i}=
 a_{\alpha_{i+1}}\Rightarrow \alpha_i<\alpha_{i+1}.
\end{equation}
Observe that $\mathbf{v}_a=(a_{\alpha_1},a_{\alpha_2},\ldots,
a_{\alpha_{k-1}})\in W_{k,q-1}$, and (\ref{E:Order}) implies that
$\alpha^{-1}$ is consistent with $\mathbf{v}_a$. We then define $
F'=F(\mathbf{a})=F(\mathbf{v}_a,\alpha^{-1})$ as the facet
corresponding to $\mathbf{a}$, and obtain that
$\overline{\mathbf{v}}_{F'}=\mathbf{a}$.
\end{proof} If
$\alpha^{-1}=\bar{\alpha}_1\bar{\alpha}_2
\ldots\bar{\alpha}_{k-1}$, then the vertices of $F(\mathbf{a})$
are
\begin{equation}\label{E:vertices}
\mathbf{v}_a^{(1)}=\mathbf{v}_a\textrm{ and }
\mathbf{v}_a^{(i+1)}=\mathbf{v}_a^{(i)}
+e_{{\bar\alpha}_{k-i}}\textrm{, for all } i=1,2,\ldots,k-1.
\end{equation}
Thus, the vertices of $F(\mathbf{a})$ are obtained by incrementing
the coordinates of $\mathbf{v}_a$ one by one, in the reverse order
of their appearance in $\mathbf{a}$.

\subsection{Shellability, order complex and barycentric subdivision}

In this paper we explore the shellability of $T_{k,q}$ and some of
its subcomplexes. Here we provide a brief overview of the concept
of shellability. More details can be found in \cite{Bj-Wa1} and
Section~8 of \cite{Ziegler}. As a good general reference for the
basic definitions and notations about simplicial complexes we
recommend \cite{Matbook} or \cite{Mun}.

 A simplicial complex $K$ is \textit{shellable} if
$K$ is pure and there exists a linear ordering
(\textit{\textit{shelling order}}) $F_1, F_2,\ldots, F_t$ of
maximal faces (facets) of $K$ such that for all $i < j\leq t$,
there exist some $l < j$ and a vertex $v$ of $F_j$, such that
\begin{equation}\label{E:defshell}
F_i\cap F_j \subseteq F_l\cap F_j =F_j\setminus\{v\}.
\end{equation}

  For a fixed shelling order $F_1, F_2,\ldots, F_t$ of $K$,
  the \textit{restriction} $\mathcal{R}(F_j)$ of the facet $F_j$ is
  defined by
$\mathcal{R}(F_j) = \{v \in F_j : F_j \setminus \{v\}\subset
F_i\textrm{ for some }1 \leq i < j\}.$ \noindent Geometrically, if
we build up $K$ from its facets according to the shelling order,
then $\mathcal{R}(F_j)$ is the unique minimal new face added at
the $j$-th step, upon gluing $F_j$. The \textit{type} of the facet
$F_j$ in the given shelling order is the cardinality of
$\mathcal{R}(F_j)$, that is, $\mathrm{type}(F_j) = |\mathcal{R}(F_j) |$.

  For a $d$-dimensional simplicial complex $K$, we denote
the number of $i$-dimensional faces of $K$ by $f_i$, and call
$f(K) = (f_{-1},f_0, f_1,\ldots ,f_d)$ the $f$-\textit{vector} of
$K$. The $h$-\textit{vector} of $K$, denoted $h(K) = (h_0,
h_1,\ldots, h_d,  h_{d+1})$, is a linear transformation of its
$f$-vector defined by
$$h_s = \sum_{i=0}^{s} (-1)^{\,s-i} \binom{d+1-i}{\,d+1-s} f_{i-1}.$$

\noindent For a shellable simplicial complex $K$,
$h_s(K)=|\{F\textrm{ is a facet of }K: \mathrm{type}(F)=s\}|$ is an
important combinatorial interpretation of $h(K)$. This
interpretation of the $h$-vector was of great significance in the
proof of the upper-bound theorem and in the characterization of
$f$-vectors of simplicial polytopes (see chapter 8 in
\cite{Ziegler}).

 For a face $\sigma$ of a simplicial complex $K$,
the link and the star of $\sigma$ are subcomplexes of $K$ defined
as $\mathrm{link}_K(\sigma)=\left\{\tau\in K: \sigma\cup \tau \in K
\textrm{ and }\sigma\cap \tau = \emptyset \right\}$ and
$\mathrm{star}_K(\sigma)=\bigcup_{\sigma\subseteq \tau}\tau.$

\noindent Note that $\mathrm{star}_K(\sigma)$ is the \emph{join} of
$\sigma$ and $\mathrm{link}_K(\sigma)$, i.e.,
$\mathrm{star}_K(\sigma)=\sigma*\mathrm{link}_K(\sigma)$.

Now, we recall some notions about the posets and poset topology.
For the basic concepts about the posets and the poset terminology
we refer to Section 3 of \cite{Stanleybook}. The \textit{order
complex }$\Delta(P)$ of a poset P is the simplicial complex on the
vertex set $P$ whose faces are the chains in $P$. This operation
(introduced by Alexandrov) provides a bridge between combinatorics
and topology, allowing one to define topological (geometric)
properties of posets. For example, a poset $P$ is shellable if and
only if $\Delta(P)$ is shellable.

We let $\hat 0$ and $\hat 1$ denote the minimal and maximal
element of a graded poset $P$, and say that
$\Delta(\overline{P})=\Delta(P\setminus \{\hat 0,\hat 1\})$ is the
\textit{reduced order complex} of $P$.
 If $P_K$ is the face
poset of a simplicial complex (or a cell complex) $K$, the
\textit{barycentric subdivision} $Sd(K)$ of $K$ is defined as
$\Delta (P_K\setminus \{\hat{0}\})$. For example, the barycentric
subdivision of $(k-1)$-simplex is $Sd(\Delta^{k-1})=\Delta
(B_k\setminus \{\emptyset\})$ (here $B_k$ denotes the Boolean
lattice), while
$Sd(\partial\Delta^{k-1})=\Delta(\overline{B}_{k})$. Notation and
conventions for order complexes follow \cite{BjoMethods, Tools}.

\subsection{Labeling of a poset}
The method of labeling the edges of the Hasse diagrams of a
certain posets was introduced by Stanley in \cite{Stalex}. This
idea was further developed in \cite{Bjolex} and \cite{BjWa}. Here
we review a variant of $R$-labeling and the $EL$-labeling of a
poset. Let $E(P)$ denote the set of all covering relations in $P$,
i.e., $E(P)=\{(x,y)\in P:x\prec y\}$. In other words, $E(P)$ is
the set of the edges in the Hasse diagram of $P$. A function
$\lambda:E(P)\to\mathbb Z$ gives an \textit{edge-labeling} of $P$.
If $\mathfrak{c}:x=x_0\prec x_1\prec\cdots\prec x_k=y$ is a
maximal chain of the interval $[x,y]$, then we write
$\lambda(\mathfrak{c})=(\lambda(x_0,x_1),\lambda(x_1,x_2),\ldots,\lambda(x_{k-1},x_k))$.

An edge-labeling $\lambda$ of $P$ is called an
$R$-\textit{labeling} if in each closed interval $[x,y]$ of $P$
there exists a unique maximal (unrefinable) chain $x=x_0\prec
x_1\prec\cdots\prec x_n=y$ such that $\lambda(x_0,x_1)\leq
\lambda(x_1,x_2)\leq \cdots \leq \lambda(x_{n-1},x_n)$. Such a
chain is called \emph{rising}.
 An
$R$-labeling $\lambda$ of $P$ is called
\textit{edge-lexicographical labeling} (an $EL$-labeling, for
short) if, in addition,
 $\lambda(\mathfrak{c})$ lexicographically precedes $\lambda(\mathfrak{c}')$ for
 every other maximal chain $\mathfrak{c}'$  of
$[x, y]$.

If a poset $P$ admits an $EL$-labeling, then the lexicographic
order of its maximal chains is a shelling of $\Delta(P)$.
Moreover, the corresponding order on the maximal chains of
$\overline{P}$ provides a shelling of $\Delta(\overline{P})$; see
\cite{Bjolex} and \cite{Bj-Wa1}.

An $R$-labeling of a graded poset $P$ can
be used to derive important enumerative characteristics of $P$
(M\" obius function, Euler characteristic, etc). In this paper, we
use this technique to obtain a combinatorial interpretation of
$h$-vector of $\Delta(P)$. Let $\lambda$ be an $R$-labeling of $P$
and let $\mathfrak{C}$ denote the set of all maximal chains in $P$.
For each maximal chain $\mathfrak{c}:\hat 0\prec x_1\prec \cdots\prec
x_{k}=\hat 1$ (the facet of $\Delta(P)$), we consider its
\textit{label} $\lambda(\mathfrak c)=(\lambda(\hat 0,x_1),\lambda(x_{
1},x_2), \ldots,\lambda(x_{k},\hat 1))$. The \textit{descent set}
$D(\mathfrak c)$ and the \textit{number of descents} $\mathrm{des}(\mathfrak c)$ of $\mathfrak c$ are
defined as $D(\mathfrak c)=\big\{i\in [k-1]:\lambda(x_{i-1},x_i)>
\lambda(x_{i},x_{i+1})\big\}\textrm{ and }\mathrm{des}(\mathfrak c)=|D(\mathfrak c)|.$ Theorem
3.14.2 in \cite{Stanleybook} provides the following combinatorial
interpretation of the entries of $h(\Delta (\overline P))$:
\begin{equation}\label{E:hprekopadova}
h_m(\Delta(\overline P))=\big|\{\mathfrak c \in \mathfrak{C}:\mathrm{des}(\mathfrak c)=m\}\big|.
\end{equation}

\begin{ex}[The $h$-vector of
$Sd(\partial \Delta^{k-1})$]\label{Ex} There is a natural labeling
of the Boolean lattice $B_k$ defined by $\lambda(A\prec
B)=B\setminus A$. This labeling is an $R$-labeling, and moreover
it is an EL-labeling. The labels of all maximal chains are all
permutations from $\mathbb{S}_k$, and the \textit{lexicographic
order} $<_L$ on $\mathbb{S}_k$
\begin{equation}\label{E:lexshell}
\pi<_L \sigma\textrm{ if and only if
}\pi_1=\sigma_1,\pi_2=\sigma_2,\ldots ,
\pi_{j-1}=\sigma_{j-1}\textrm{ and } \pi_j<\sigma_j
\end{equation} will define a
shelling order (see Theorem 2.3 in \cite{Bjolex}) for $Sd(\partial
\Delta^{k-1})$. From (\ref{E:hprekopadova}) it follows that the
entries of $h$-vector of $Sd(\partial \Delta^{k-1})$ are the
well-known \textit{Eulerian numbers} $A(k,i)$. These numbers count
the number of permutations in $\mathbb{S}_k$ with exactly $i$
descents:
$$h_i( Sd(\partial\Delta^{k-1}))=
A(k,i)=\big|\{\pi\in \mathbb{S}_k: \mathrm{des}(\pi)=i\}\big|.$$ The
Eulerian numbers satisfy the following recursive relations:
\begin{equation}\label{E:Euler}
A(k,i)=(i+1)A(k-1,i)+(k-i)A(k-1,i-1).
\end{equation}
\end{ex}

%%%%%%%%%%%%%%%%%%%%%%%%%%%%%%%%%%%%%%%%%%%%%%%%%%%
\section{Partitions and products of chains}
Let $C_m=([m]\cup \{0\},\leq)$ denote the chain of length $m$, and
let $P_{\lambda_1,\lambda_2,\ldots,\lambda_s}$ denote the direct
product of $s$ chains of lengths
$\lambda_1,\lambda_2,\ldots,\lambda_s$. For example, the $k$-fold
product of $C_1$ is $P_{1,1,\ldots,1}=C_1^{\times k}\cong B_k$.
The following statement is well known, but we include the proof
here for the sake of completeness and for future applications.
\begin{prop}\label{P:proizvod je R}
The product of chains admits an $R$-labeling.
\end{prop}
\begin{proof} Assume that $P_{\lambda_1,\lambda_2,\ldots,\lambda_s}=
C_{\lambda_1}\times C_{\lambda_2}\times \cdots\times
C_{\lambda_s}$. For all $i=1,2,\ldots,s$, we label all of the
edges of $C_i$ with $i$, and transfer these labels to
$P_{\lambda_1,\lambda_2,\ldots,\lambda_s}$. For $
\mathbf{x}=(x_1,x_2,\ldots,x_s)$ and
$\mathbf{y}=(y_1,y_2,\ldots,y_s)$ such that $\mathbf{x} \prec_P
\mathbf{y}$, we let $\lambda(\mathbf{x}\prec_P \mathbf{y})=i$ if
and only if $x_i\prec_{C_i}y_i$ (the unique coordinate where
$\mathbf{x}$ and $\mathbf{y}$ are different). For any interval
$[\mathbf{x},\mathbf{y}]$ in
$P_{\lambda_1,\lambda_2,\ldots,\lambda_s}$ we have that
$[\mathbf{x},\mathbf{y}] \cong
[x_1,y_1]\times[x_2,y_2]\times\cdots \times [x_s,y_s].$ In our
labeling, the unique rising chain from $\mathbf{x}$ to
$\mathbf{y}$ in $P_{\lambda_1,\lambda_2,\ldots,\lambda_s}$ is
$\mathbf{x}=\mathbf{z}_0\prec_P \cdots\prec_P \mathbf{z}_1
 \prec_P\cdots\prec_P \mathbf{z}_i\prec_P \cdots
 \prec_P \mathbf{z}_s=\mathbf{y}$,
 where $\mathbf{z}_i=(y_1,\ldots,y_i,x_{i+1},\ldots, x_s)$.
In this chain, all the edges between $\mathbf{z}_{i-1}$
 and $\mathbf{z}_i$ are labeled by $i$.
\end{proof}
Since the order of the chains (the factors of the product)
defining $P_{\lambda_1,\lambda_2,\ldots,\lambda_s}$ does not
affect the combinatorial or topological
  properties of this poset, we may assume that $\lambda_1\geq
\lambda_2\geq\cdots\geq \lambda_s$ forms a partition of an
integer. We let $\mathrm{Par}(k)$ denote the set of all partitions of $k$
and let $p_k=|\mathrm{Par}(k)|$. Further, let $\mathrm{Par}(k,s)$ denote the set of
all partitions of $k$ into exactly $s$ parts (summands), and let
$p_{k,s}=|\mathrm{Par}(k,s)|$. A partition
$\bm\lambda=(\lambda_1,\lambda_2,\ldots,\lambda_s)$ of $k$ into
$s$ parts can also be encoded as
$(n^{m_1}_1,n^{m_2}_2,\ldots,n^{m_t}_t)$, where $m_i$ counts the
multiplicity of $n_i$: $m_1+m_2+\cdots+m_t=s\textrm{ and } n_1m_1+
n_2m_2+\cdots+n_tm_t=k$.
\begin{defn}For a partition
$\bm{\lambda}=(\lambda_1,\ldots,\lambda_s)$ of $k$, we
define $K_\lambda=\Delta(\overline{P}_\lambda)$, where
$P_{\lambda}=P_{\lambda_1,\lambda_2,\ldots,\lambda_s}=
C_{\lambda_1}\times C_{\lambda_2}\cdots\times C_{\lambda_s}.$
\end{defn}
By the preceding definition, the set of partitions of $k$ can be
regarded geometrically as the family of simplicial complexes
$\mathcal{C}_k=\{K_\lambda\}_{\lambda\in \mathrm{Par}(k)}$. For any
$\bm{\lambda}\in \mathrm{Par}(k)$ the complex $K_\lambda$ is
$(k-2)$-dimensional. For instance, $K_{(k)}\cong \Delta^{k-2}$,
while $K_{(1,1,\ldots,1)}\cong Sd(\partial \Delta^{k-1})$. The
numbers of vertices and facets of $K_\lambda$ are
\begin{equation}\label{E:facets}
f_0(K_\lambda)=(\lambda_1+1)
(\lambda_2+1)\cdots(\lambda_s+1)-2\textrm{ and }
f_{k-2}(K_\lambda)=\frac{k!}{\lambda_1!\lambda_2!\cdots\lambda_s!}.
\end{equation}

\noindent Recall that the labels of maximal chains of
$P_{\lambda}$ (as described in Proposition \ref{P:proizvod je R}),
correspond to $\mathbb{S}_{M_\lambda}$, the set of all
permutations of the multiset
$M_\lambda=\{1^{\lambda_1},2^{\lambda_2}, \ldots,s^{\lambda_s}\}$.
These permutations encode the facets of the simplicial complex
$K_\lambda$. Equivalently, the facets of $K_\lambda$ may be
identified with the permutations of $[k]$ in which the numbers
from each of $s$ sets $L_1=\{1,\ldots,\lambda_1\}$,
$L_2=\{\lambda_1+1,\ldots,\lambda_1+\lambda_2\}$, $\ldots,
L_s=\{k-\lambda_s+1,\ldots,k\}$ appear in their natural order.
Simply, the edges of each $C_i$ can be labeled by the elements
from $L_i$ and the elements of $P_\lambda$ can be identified with
the subsets of $[k]$. Consequently, each $K_\lambda\in
\mathcal{C}_k$ can be regarded as a subcomplex of $Sd(\partial
\Delta^{k-1})$. To describe the containment relation between two
complexes in $\mathcal{C}_k$, we introduce a coarsening of the
refinement order on $\mathrm{Par}(k)$ defined in \cite{Ziegler-part}. For
two partitions $\bm{\lambda} = (\lambda_1, \ldots, \lambda_s)$ and
$\bm{\mu} = (\mu_1, \ldots, \mu_t)$ of $k$, we write $\bm{\lambda}
<_c \bm{\mu}$ if each part of $\bm{\mu}$ is obtained by merging
consecutive parts of $\bm{\lambda}$.
\begin{cor}
For any two partitions $\bm\lambda, \bm\mu\in \mathrm{Par}(k)$, we have
$K_\mu \subset K_\lambda$ if and only if $\bm\lambda <_c\bm\mu$.
\end{cor}

To the best of our knowledge, this is a new order on
$\mathrm{Par}(k)$. We now describe some of the properties of
$P^{c}_{k}=(\mathrm{Par}(k),<_c)$. Both posets $(\mathrm{Par}(k),<)$ and
$P^{c}_{k}=(\mathrm{Par}(k),<_c)$ are graded of rank $k-1$, and the
cardinality of the $i$-th level in both is $p_{k,i}$. However,
there is significantly less covering in $P^{c}_{k}$.

 Ziegler showed in \cite{Ziegler-part} that $(\mathrm{Par}(k),<)$ is not Cohen-Macaulay and not shellable
 for $k\ge 19$, thus giving a negative answer to the question raised by Bj\"
 orner in \cite{Bjolex}.

\begin{thm} The poset $P^{c}_{k}$ is shellable and
Cohen-Macaulay for all $k\in \mathbb{N}$.
\end{thm}
\begin{proof}
There is a natural bijection between $P^{c}_{k}$ and a subposet of
the Boolean lattice $B_{k-1}$:
$$ \bm{\lambda} = (\lambda_1, \ldots, \lambda_s)
\longmapsto A_{\lambda}=[k]\setminus
\{\lambda_1,\lambda_1+\lambda_2,\ldots,\lambda_1+\lambda_2+\cdots+\lambda_s\}.$$
Then $\bm \lambda \prec_c \bm \mu$ if and only if
$A_{\lambda}\subset A_{\mu}$. More specifically, if
$\mu_i=\lambda_i+\lambda_{i+1}$, then
$A_{\mu}=A_{\lambda}\cup\{\lambda_1+\lambda_2+\cdots+
\lambda_{i}\}$. This defines an $EL$-labeling
$$\ell(\bm \lambda\prec_c \bm\mu)=A_{\mu}\setminus A_{\lambda}=\lambda_1+\lambda_2+\cdots+
\lambda_{i},$$ which is a restriction of the standard
lexicographic labeling of $B_{k-1}$.
\end{proof}
 Therefore, the reduced order complex $\Delta
(\overline{{P}^{c}_{k}})$ is a shellable $(k-3)$-dimensional disk
and its facets (given by the labels of maximal chains) correspond
to certain permutations in $\mathbb{S}_{k-1}$.

We use the shelling induced by the edge-labeling from
Proposition~\ref{P:proizvod je R} to compute the $h$-vector of
complexes in $\mathcal{C}_k$.

\begin{prop}\label{P:hrekrel}
Let $\bm{\lambda}=(\lambda_1, \dots, \lambda_s)\in \mathrm{Par}(k)$ and
$\bm{\lambda}'=(\lambda_1, \dots, \lambda_{s-1})\in
\mathrm{Par}(k-\lambda_s)$. Then
\begin{equation}\label{E:rrforh}
h_{i}(K_\lambda)= \sum_{j=0}^{\lambda_s} {k-\lambda_s-i+j\choose
j}{i+\lambda_s-j\choose \lambda_s-j}h_{i-j}(K_{\lambda'}).
\end{equation}
\end{prop}

\begin{proof}
Consider the multisets
$M_\lambda=\{1^{\lambda_1},\dots,s^{\lambda_s}\}$ and
$M_{\lambda'}=\{1^{\lambda_1},\dots,(s-1)^{\lambda_{s-1}}\}$. By
(\ref{E:hprekopadova}) and Proposition~\ref{P:proizvod je R},
$h_i(\Delta(K_\lambda))$ counts permutations of $M_\lambda$ with
exactly $i$ strict descents.

Any permutation of $M_\lambda$ is obtained by inserting
$\lambda_s$ copies of $s$ into a permutation of $M_{\lambda'}$. If
a permutation of $M_{\lambda'}$ has $i-j$ descents, there are
$k-\lambda_s-i+j$ positions to insert $j$ copies of $s$ between
non-descending pairs or at the beginning, and the remaining
$\lambda_s-j$ copies can go among the $i$ descent positions or at
the end. Summing over $j$ yields (\ref{E:rrforh}).
\end{proof}
\noindent Note that for $\lambda_s=1$ the formula (\ref{E:rrforh})
is essentially the same as (\ref{E:Euler}).
\begin{cor} For any
$\bm{\lambda}=(\lambda_1,\lambda_2,\ldots,\lambda_s)\in \mathrm{Par}(k)$ we
have that
$$h_{k-\lambda_1}(K_\lambda)={\lambda_1 \choose \lambda_2}
{\lambda_1\choose \lambda_3}\cdots{\lambda_1\choose \lambda_s},
\textrm{ and  } h_j(K_\lambda)=0\textrm{ for all }j>k-\lambda_1.$$
If $\bm\lambda=(\lambda_1,k-\lambda_1)$ has exactly two parts,
then $h_i(K_\lambda)={\lambda_1\choose i}{k-\lambda_1\choose i}$
for all $i$, $0\leq i\leq k-\lambda_1$.
\end{cor}

\begin{lem}\label{L:JI} If $\bm{\lambda}$ has more
than one part (i.e., $\bm\lambda\neq(k)$), then it is not possible
to find two non-empty complexes $M$ and $N$ such that
$K_\lambda\cong M*N$.
\end{lem}
We say that for $\bm\lambda\neq(k)$, the complex $K_\lambda$ is
\textit{join-irreducible}.
\begin{proof}
The vertices of $K_\lambda$ can be written as the union of
independent sets $A_i=\{\mathbf{x}\in
\overline{P}_\lambda:x_1+x_2+\cdots+x_k=i\}$ (the levels in
$P_\lambda$). If $K_\lambda\cong M*N$, then we have that
$A_i\subset M$ or $A_i\subset N$. Moreover, since the subcomplex
of $K_\lambda$ induced
 by $A_i\cup A_{i+1}$ is not a complete bipartite
graph, it follows that all the sets $A_i$ must be contained within
the same complex (either $M$ or $N$), while the other complex is
empty.
\end{proof}

\begin{thm}\label{T:razparrazkom}
The family $\mathcal{C}_k=\{K_\lambda\}_{\lambda\in
\mathrm{Par}(k)}$ consists of $p_k-1$ pairwise nonisomorphic
shellable $(k-2)$-discs and a single $(k-2)$-sphere
$K_{(1,1,\ldots,1)}\cong Sd(\partial \Delta^{k-1})$.
\end{thm}
\begin{proof}
We proceed by induction on $k$ and Lemma \ref{L:JI} to prove that
 different partitions of $k$ define non-isomorphic complexes.
 For a partition $\bm\lambda=(\lambda_1,\lambda_2,\ldots,\lambda_t)\in
\mathrm{Par}(k)$, the link of a vertex
$\mathbf{a}=(a_1,\ldots,a_t)$ in $K_\lambda$ is
$\Delta\left(\overline{[\hat 0,\mathbf{a}]_{P_\lambda}}\right)*
\Delta\left(\overline{[\mathbf{a},\hat 1]_{P_\lambda}}\right)$, and one of
these complexes may be empty.

For all $i=1,2,\dots,t$, let $\bm\lambda^{(i)}$ denote the
partition obtained from $\bm \lambda$ by decreasing its $i$-th
part $\lambda_i$ by one. The links of atoms in $P_\lambda$ are
precisely the complexes $K_1, K_2, \ldots, K_s$, where
$K_i=\Delta(\overline{P}_{\lambda^{(i)}})$.
 If $K_\lambda\cong K_\mu$, then the links of their corresponding
vertices are isomorphic, so $K_1, K_2, \ldots, K_s$ must also
occur as the links of vertices in $K_\mu$. By the inductive
hypothesis, this forces $\bm\lambda=\bm\mu$.
\end{proof}

%%%%%%%%%%%%%%%%%%%%%%%%%%%%%%%%%%%%%%%%%%%%%%%%%

\section{The star and the link of a vertex in $T_{k,q}$}

In this section, we describe the link and the star of a vertex
$\mathbf{v}\in W_{k,q}$. Note that two vertices
$\mathbf{x},\mathbf{y}\in W_{k,q}$ define an edge
$\mathbf{x}\mathbf{y}$ of $T_{k,q}$ if and only if
\begin{equation}\label{E:istasrana}
\mathbf{y}-\mathbf{x}=(y_1-x_1,\ldots,y_{k-1}-x_{k-1})
\in\big\{0,1\big\}^{k-1}\text{ or }
\mathbf{y}-\mathbf{x}\in\{-1,0\}^{k-1}.
\end{equation}
\begin{thm}\label{T:linkunutr} If $\mathbf{v}\in W_{k,q}$
is an interior vertex of $R_{k,q}$, then
$\mathrm{link}_{T_{k,q}}(\mathbf{v})\cong Sd
\big(\partial\Delta^{k-1}\big).$
\end{thm}
\begin{proof}
A vertex $\mathbf{v}=(v_1,\ldots,v_{k-1})$ is contained in the
interior of $R_{k,q}$ if and only if $0< v_1<v_2<\cdots<
v_{k-1}<q$. From (\ref{E:istasrana}), we conclude that
$\mathrm{link}_{T_{k,q}} (\mathbf{v})$ contains $2^k-2$ vertices of the
form
$$\mathbf{v}\pm \sum_{i\in S\subset [k-1]}e_i, \textrm{ for all }
S\neq \emptyset.$$

\noindent Recall that the vertices of $Sd
\big(\partial\Delta^{k-1}\big)$ are all subsets $A\subset [k]$
such that $A \neq \emptyset$ and $ A\neq [k]$. The facets of $Sd
\big(\partial\Delta^{k-1}\big)$ are encoded by all permutations in
$\mathbb{S}_k$, and the facet indexed by $\pi =\pi_1
\pi_2\ldots\pi_k$ contains the $k-1$ vertices $\big\{\pi_1\big\}$,
$\big\{\pi_1,\pi_2\big\}$,$\ldots,\big\{\pi_1,\pi_2,
\ldots,\pi_{k-1}\big\}$. The following map sends the vertices of
$Sd (\partial\Delta^{k-1})$ to the vertices of
$\mathrm{link}_{T_{k,q}}(\mathbf{v})$:
$$\textrm{ for }A\subset
[k], A \neq \emptyset\textrm{
and }A\neq [k],\textrm{ we let }\mathbf{v}_A=\left\{%
\begin{array}{ll}
    \mathbf{v}+\sum_{i\in A} e_i, & \hbox{if $k\notin A$;} \\
    \mathbf{v}-\sum_{i\in A^c} e_i, & \hbox{if $k\in A$.} \\
\end{array}%
\right.    $$  It is easy to conclude that the map $A \mapsto
\mathbf{v}_A$ is a bijection. We also define a bijection between
the facets of $Sd \big(\partial \Delta^{k-1}\big)$ and the facets
of
$\mathrm{link}_{T_{k,q}} (\mathbf{v})$:\\
 a facet of $Sd \big(\partial \Delta^{k-1}\big)$ defined by $\pi
=\pi_1\ldots \pi_{i-1}\,\, k\,\,\pi_{i+1}\ldots\pi_k\in
\mathbb{S}_{k}$ corresponds to the facet
$F(\mathbf{v}_\pi,\pi')\setminus \{\mathbf{v}\}$ of
$\mathrm{link}_{T_{k,q}}(\mathbf{v})$, where
\begin{equation}\label{E:stranaulinku}
\mathbf{v}_\pi=\mathbf{v}-e_{\pi_k}-\cdots-e_{\pi_{i+1}} \textrm{
and }\pi'= \pi_{i-1}\ldots\pi_1 \, \pi_{k} \ldots\pi_{i+1} \in
\mathbb{S}_{k-1}.
\end{equation}
\noindent The vertices of $F(\mathbf{v}_\pi,\pi')$ are
 $\mathbf{v}_\pi=\mathbf{v}-e_{\pi_k}-\cdots-e_{\pi_{i+1}}, \ldots,
 \mathbf{v}-e_{\pi_{k}},$ $
 \mathbf{v},\mathbf{v}+e_{\pi_{1}},\dots,
 \mathbf{v}+e_{\pi_{i-1}}+\cdots+e_{\pi_1},$
and the above defined bijections preserve the vertex-facet
incidence.
\end{proof}
\begin{rem}\label{R:star} Let $\mathbf{v}\in W_{k,q}$ be a vertex
lying in the interior of $R_{k,q}$. For a permutation
$\pi=\pi_1\ldots\pi_{i-1}k\,\pi_{i+1}\ldots\pi_k \in
\mathbb{S}_{k}$, we define
$\mathbf{a}_\pi=(v_{\pi_{i-1}},\ldots,v_{\pi_1},
v_{\pi_{k}}-1,\ldots,v_{\pi_{i+1}}-1) \in \mathcal{S}_{k,q}.$ If
we use the notation for the facets of $T_{k,q}$ defined in
Proposition \ref{P:facetsarestrings}, then it follows from
(\ref{E:stranaulinku}) that $\mathrm{star}_{T_{k,q}}(\mathbf{v})=
\bigcup_{\pi\in \mathbb{S}_k} F(\mathbf{a}_\pi).$ Therefore, the
facets of $\mathrm{star}_{T_{k,q}}(\mathbf{v})$ correspond to permutations
from $\mathbb{S}_k$. Two facets $F(\mathbf{a}_\pi)$ and
$F(\mathbf{a}_\sigma)$ share a common $(k-2)$-face if and only if
$\pi$ is covered by $\sigma$ (or vice versa) in the weak Bruhat
order on $\mathbb{S}_k$.
\end{rem}

A vertex $\mathbf{v}\in W_{k,q}$ is on the boundary of $R_{k,q}$
if some of its coordinates repeat, or if some of them are equal
$0$ or $q$. We encode the occurrences of $0$ and $q$, as well as
the multiplicities of the coordinates of $\mathbf{v}$, using its
\textit{type}:
$$\mathrm{type}(\mathbf{v})=\big(\alpha_0;\alpha_1,\ldots,
\alpha_{s-1};\alpha_s\big) \textrm{; }\alpha_0,\alpha_s\in
\mathbb{N}_0, \alpha_i\in \mathbb{N}\textrm{ for all }
i\in[s-1],$$

where the first $\alpha_0$ coordinates of $\mathbf{v}$ equal $0$,
for each $i=1,\dots,s-1$ the next $\alpha_i$ coordinates are
equal, and the last $\alpha_s$ coordinates equal $q$.

 We now relate the links of
the vertices of $T_{k,q}$ to posets that are products of chains.
\begin{thm}\label{T:linkjeproduct}Let $\mathbf{v}$ be a vertex of $T_{k,q}$
with $\mathrm{type}(\mathbf{v})=
(\alpha_{0};\alpha_1,\ldots,\alpha_{s-1};\alpha_s)$. If
$P_\mathbf{v}$ is the product of $s$ chains of lengths
$\alpha_0+\alpha_s+1$,$\alpha_1,\ldots,\alpha_{s-1}$,
 then $\mathrm{link}_{T_{k,q}}(\mathbf{v})\cong
\Delta(\overline{P}_\mathbf{v}).$
\end{thm}
\begin{proof}
Note that the bijection $\pi \mapsto F(v_\pi,\pi')$ between
$\mathbb{S}_k$ and the facets of $\mathrm{link}_{T_{k,q}}(\mathbf{v})$
described in (\ref{E:stranaulinku}) (see the proof of Theorem
\ref{T:linkunutr}),
 does not hold for all $\pi\in \mathbb{S}_k$.
The $(k-1)$-tuple $
\mathbf{v}_\pi=\mathbf{v}-e_{\pi_k}-\cdots-e_{\pi_{i+1}}$ defined
by $\pi=\pi_1\ldots \pi_{i-1}\, k\, \pi_{i+1}\ldots\pi_k\in
\mathbb{S}_k$ is a vertex of $T_{k,q}$, and
$\pi'=\pi_{i-1}\ldots\pi_1 \pi_{k}\ldots\pi_{i+1}\in
\mathbb{S}_{k-1} $ is consistent with $\mathbf{v}_\pi$ if and only
if $\pi$ satisfies the following conditions:
\begin{enumerate}
\item[(1)] If $\alpha_0>0$, the numbers $1, \dots, \alpha_0$ must
precede $k$ in $\pi$ (since the corresponding coordinates of
$\mathbf{v}$ cannot be decreased), and appear in reverse order.
\item[(2)] For $i=1,\dots,s-1$, the $\alpha_i$ consecutive
coordinates of $\mathbf{v}$ are equal, and indices of these
coordinates in $\pi$ must appear in reverse order. There are no
restrictions on the position of $k$. \item[(3)] If $\alpha_s>0$,
the numbers $k-\alpha_s, \dots, k-1$ must follow $k$ in $\pi$,
also in reverse order.
\end{enumerate}
\begin{figure}[h]
\begin{center}

\begin{tikzpicture}[scale=1.16]
\draw (0,0) -- (4,4) -- (2,6) -- (-2,2) -- (0,0);

\draw (-1,1) -- (3,5);\draw (1,1) -- (-1,3);\draw (2,2) -- (0,4);
\draw (3,3) -- (1,5);

\draw (0,1) -- (4,5) -- (2,7) -- (-2,3)--(0,1);\draw (-1,2) --
(3,6);\draw (1,2) -- (-1,4);\draw (2,3) -- (0,5);\draw (3,4) --
(1,6);

\draw (0,0)--(0,1);\draw (1,1)--(1,2);\draw (2,2)--(2,3);\draw
(3,3)--(3,4);\draw (4,4)--(4,5);

\draw (-1,1)--(-1,2);\draw (0,2)--(0,3);\draw (1,3)--(1,4);\draw
(2,4)--(2,5);\draw (3,5)--(3,6);

\draw (-2,2)--(-2,3);\draw (-1,3)--(-1,4);\draw (0,4)--(0,5);\draw
(1,5)--(1,6);\draw (2,6)--(2,7);

\draw[thick](0,0)--(0,1);\draw[thick](0,0)--(4,4);
\draw[thick](0,0)--(-2,2);

\draw[thick](-4.8,0)--(-4.8,5.6);\draw[thick](-4,0)--(-4,2.8);
\draw[thick](-3.2,0)--(-3.2,1.4);

\draw [fill] (0,0) circle [radius=0.051]; \foreach \y in
{0,...,4}{\draw [fill] (-4.8,1.4*\y) circle [radius=0.051];}
\foreach \y in {0,...,2}{\draw [fill] (-4,1.4*\y) circle
[radius=0.051];} \foreach \y in {0,...,1}{\draw [fill]
(-3.2,1.4*\y) circle [radius=0.051];}

\foreach \x in {0,...,4}{\draw [fill] (\x,\x) circle
[radius=0.051];}

\foreach \x in {0,...,4}{\draw [fill] (\x-1,\x+1) circle
[radius=0.051];} \foreach \x in {0,...,4}{\draw [fill] (\x-2,\x+2)
circle [radius=0.051];}

\foreach \x in {0,...,4}{\draw [fill] (\x,1+\x) circle
[radius=0.051];}

\foreach \x in {0,...,4}{\draw [fill] (\x-1,\x+2) circle
[radius=0.051];} \foreach \x in {0,...,4}{\draw [fill] (\x-2,\x+3)
circle [radius=0.051];}

\node at (-5,.7) {$2$};\node at (-5,2.2) {$1$};\node at (-5,3.6)
{$7$};\node at (-5,5) {$6$};\node at (-4.2,.7) {$4$};\node at
(-4.2,2.2) {$3$};\node at (-3,.7) {$5$}; \node at (-5,-.3)
{$C_4$};\node at (-4.1,-.3) {$C_2$};\node at (-3.1,-.3) {$C_1$};

\node at (.6,.4) {$2$};\node at (1.6,1.4) {$1$};\node at (2.6,2.4)
{$7$};\node at (3.6,3.4) {$6$};\node at (-.6,.4) {$4$};\node at
(-1.6,1.4) {$3$};\node at (-0.12,.5) {$5$}; \node at (0,-.3)
{$_{(0,0,1,1,2,q)}$};\node at (-1.81,.83) {$_{(0,0,1,2,2,q)}$};
\node at (-2.8,1.8) {$_{(0,0,2,2,2,q)}$}; \node at (-2.89,2.92)
{$_{(0,0,2,2,3,q)}$};

\node at (-1.88,4) {$_{(0,1,2,2,3,q)}$};\node at (-0.88,5)
{$_{(1,1,2,2,3,q)}$};\node at (0.01,6.11) {$_{(0,0,1,1,2,q-1)}$};
\node at (1.9,.9) {$_{(0,1,1,1,2,q)}$};\node at (2.9,1.9)
{$_{(1,1,1,1,2,q)}$};\node at (4,2.8) {$_{(0,0,0,0,1,q-1)}$};
\node at (4.9,3.9) {$_{(0,0,0,0,1,q)}$};\node at (4.9,4.9)
{$_{(0,0,0,0,2,q)}$};\node at (3.9,6) {$_{(0,0,0,1,2,q)}$}; \node
at (2.9,7) {$_{(0,0,1,1,2,q)}$};

\end{tikzpicture}
\caption{The poset $P_{4,2,1}$ describes the link of
$v=(0,0,1,1,2,q)\in W_{7,q}$}\label{fig:linktjemena}
\end{center}
\end{figure}
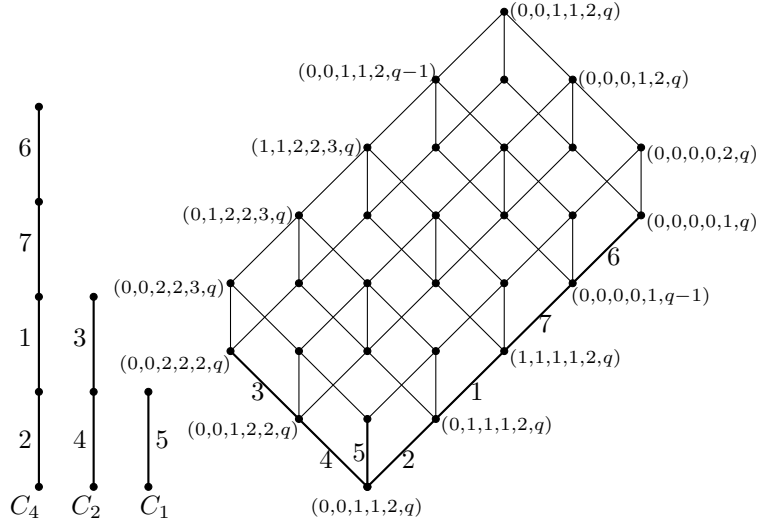
\noindent Now, we define the poset
$P_\mathbf{v}=P_{\alpha_0+\alpha_s+1,\alpha_1,
\ldots,\alpha_{s-1}}=C_{\alpha_0+\alpha_s+1}\times
C_{\alpha_1}\times\cdots\times C_{\alpha_{s-1}}$ as the product of
chains whose lengths are prescribed by $\mathrm{type}(\mathbf{v})$.

 \noindent We label the
edges of $C_{\alpha_0+\alpha_s+1}$ with
$\alpha_0,\alpha_{0}-1,\ldots, 2,1,k,k-1\ldots,k-\alpha_s$, from
the minimum to the maximum of $C_{\alpha_0+\alpha_s+1}$ in this
order, see Figure \ref{fig:linktjemena}. Similarly, the labels of
the edges in the chain $C_{\alpha_i}$ of length $\alpha_i$, from
the bottom to the top, are
$$\alpha_0+\alpha_1+\cdots+\alpha_i,
\alpha_0+\alpha_1+ \cdots+\alpha_i-1,\ldots,
\alpha_0+\alpha_1+\cdots+\alpha_{i-1}+1.$$ We then transfer these
labels to the edges of the Hasse diagram of $P_\mathbf{v}$. If
$\mathbf{x}=(x_1,x_2,\ldots,x_s)$ is covered by
$\mathbf{y}=(y_1,y_2,\ldots,y_s)$ in $P_\mathbf{v}$ (here
$x_j=y_j$ for all $j\neq j_0$ and $x_{j_0}\prec_{C_{j_0}}y_{j_0}$
in $C_{j_0}$), we label $\mathbf{x}\prec_{P_\mathbf{v}}
\mathbf{y}$ in the same way as the edge $x_{j_0}\prec_{C_{j_0}}
y_{j_0}$ is labeled in $C_{j_0}$, see Figure
\ref{fig:linktjemena}. This labeling should not be confused with
an $R$-labeling.

The facets of $\mathrm{link}_{T_{k,q}}(\mathbf{v})$ are in bijection (under
the same map described in (\ref{E:stranaulinku})) with the
permutations from $\mathbb{S}_k$ that satisfy conditions
$(1)-(3)$. We can recognize all of these permutations as the
labels of maximal chains in $P_\mathbf{v}$, and therefore
\mbox{$\mathrm{link}_{T_{k,q}}(\mathbf{v})\cong
\Delta(\overline{P}_\mathbf{v})$.}

Also, there is a natural bijection between the vertices of
$\mathrm{link}_{T_{k,q}}(\mathbf{v})$ and the elements of $P_\mathbf{v}$.
The minimum and the maximum of $P_\mathbf{v}$ correspond to
$\mathbf{v}$. If $\mathbf{x}\in P_\mathbf{v}$
 corresponds to a vertex $\mathbf{w}\in
W_{k,q}$ and $\mathbf{x}\prec_P \mathbf{y}$ is labeled by $i$,
then $\mathbf{y}$ corresponds to $\mathbf{w}+e_i$ (if $i\neq k$)
or $\mathbf{w}-\mathbbm{1}$ if $i=k$, see Figure
\ref{fig:linktjemena}.
\end{proof}
The subset of $\mathbb{S}_k$ containing all permutations that
satisfy conditions $(1)-(3)$ describes the facets of
$\mathrm{star}_{T_{k,q}}(\mathbf{v})$.

\begin{rem}\label{R:linkjepart}The type of a vertex
$\mathbf{v}$ is determined by its \emph{support} in $R_{k,q}$, the
minimal face containing $\mathbf{v}$. For an $(s+1)$-tuple
$\bm{\alpha}=(\alpha_0; \alpha_1,
 \dots, \alpha_{s-1}; \alpha_s)$ with $\alpha_i>0$
for $i\neq 0,s$ and $\alpha_0+\dots+\alpha_s=k-1$, the
$(s-1)$-face
\begin{equation}\label{E:facetype}
F_\alpha=\conv\left\{\mathbf{w}_{1+\alpha_s},
\mathbf{w}_{1+\alpha_s+\alpha_{s-1}},\ldots,
\mathbf{w}_{1+\alpha_s+\cdots+\alpha_{1}}\right\}
\end{equation}
 is the unique minimal face that contains all vertices of $T_{k,q}$
  of type $\bm{\alpha}$ in its
relative interior. This correspondence defines a bijection between
vertex types and faces of $R_{k,q}$. From the definition of the
edgewise triangulation, such a face contains a vertex from
$W_{k,q}$ in its relative interior if and only if $s\le q$. The
induced triangulation on $F_\alpha$ is isomorphic to $T_{s,q}$;
hence, an $(s-1)$-face of $T_{k,q}$ contains ${q-1 \choose s-1}$
internal vertices.

\end{rem}

Now, we recognize the links of vertices of $T_{k,q}$ as the
complexes from $\mathcal{C}_k$.
\begin{thm}\label{T:typeoflinks} If $q\geq k$, then the links of
all vertices of $T_{k,q}$ are exactly all complexes from
$\mathcal{C}_k$. If $q<k$, we have that
$\big\{\mathrm{link}_{T_{k,q}}(\mathbf{v}):\mathbf{v}\in W_{k,q}\big\}
=\big\{K_\lambda: \lambda\in \mathrm{Par}(k,s)\textrm{ for }s\leq q
\big\}$.
\end{thm}

\begin{proof}
Note that a vertex $\mathbf{v}$ whose type is $\mathrm{type}(\mathbf{v})=
\big(\alpha_0;\alpha_1,\ldots,\alpha_{s-1};\alpha_s\big)$ defines
a partition
$\bm{\lambda}(\mathbf{v})=(\lambda_1,\lambda_2,\ldots,\lambda_s)$
of $k$ into exactly $s$ parts (each $\lambda_i$ is equal to some
$\alpha_j$ or $\alpha_0+\alpha_s+1$). From Theorem
\ref{T:linkjeproduct} it follows that
$\mathrm{link}_{T_{k,q}}(\mathbf{v})\cong K_{\lambda(v)}\in \mathcal{C}_k$.
If $q\geq s$, then for any partition
$\bm{\mu}=(\mu_1,\mu_2,\ldots,\mu_s)$ of $k$ into $s$ parts we
consider the vertex
$$\mathbf{v}_\mu=(\underbrace{0,\ldots,0}_{\mu_1-1},
\underbrace{1,\ldots,1}_{\mu_2},\ldots,
\underbrace{s-1,\ldots,s-1}_{\mu_s})\in W_{k,q},$$ and obtain that
$\mathrm{link}_{T_{k,q}}(\mathbf{v}_\mu)\cong K_\mu$. The link of a vertex
lying in the relative interior of an $(s-1)$-dimensional face
corresponds to a partition $\bm \mu \in \mathrm{Par}(k,s)$. When
$q< k$, the links of the vertices of $T_{k,q}$ are precisely the
complexes in $\mathcal{C}_k$ corresponding to partitions of $k$
into at most $q$ parts.
\end{proof}

\noindent For example, the link of a vertex of $T_{k,q}$ lying in
the interior of $R_{k,q}$ corresponds to the partition
$(1,1,\ldots,1)$, see Theorem \ref{T:linkunutr}. The link of a
vertex of $R_{k,q}$ (it is a $(k-2)$-simplex) corresponds to the
partition $(k)$, while there are $\lfloor \frac{k}{2}\rfloor$
combinatorial types of links of vertices that are on the edges of
$R_{k,q}$ (correspond to the partitions of $k$ into two parts).

\begin{cor}
The number of combinatorially distinct links of the vertices of
$T_{k,q}$ is $p_k$ if $q\geq k$, and
$p_{k,1}+p_{k,2}+\cdots+p_{k,q}$ if $q<k$. The $h$-vector of
$\mathrm{link}_{T_{k,q}}(\mathbf{v})$ can be calculated by the formula
(\ref{E:rrforh}), see Proposition \ref{P:hrekrel}.
\end{cor}

Vertices in the relative interior of the same $(s-1)$-dimensional
face will have isomorphic links. The interior vertices of two
distinct $(s-1)$-faces $F=\{w_{i_1},\ldots,w_{i_s}\}$ and
$F'=\{w_{i'_1},\ldots,w_{i'_s}\}$ of $T_{k,q}$ will have
isomorphic links if and only if their types define the same
partition of $k$, that is, when
$$\{i_2-i_1, \ldots,i_s-i_{s-1},k-i_s+i_1-1\}=
\{i'_2-i'_1,\ldots,i'_s-i'_{s-1},k-i'_s+i'_1-1\}.$$

We now determine the number of faces whose interior vertices share
a prescribed link.
\begin{prop}\label{P:faceswithprescribed partition}
Let $\bm{\beta}=(n^{m_1}_1,n^{m_2}_2,\ldots,n^{m_t}_t)$ be a
partition of $k$ into $s$ parts. If $q\geq s$, then the number of
$(s-1)$-dimensional faces of $R_{k,q}$ whose interior vertices
have a link combinatorially equivalent to $K_\beta$ is
$$\frac{k\cdot (s-1)!}{m_1! m_2! \cdots m_t!}.$$
\end{prop}

\begin{proof} By Remark \ref{R:linkjepart}, there is a bijection between the
$(s-1)$-dimensional faces of $R_{k,q}$ and the types
$(\alpha_0;\alpha_1,\ldots,\alpha_{s-1};\alpha_s)$ of the vertices
in their relative interiors. To obtain the desired count, we need
to determine how many such types correspond to the given partition
$\bm{\beta}$.

Observe that $\alpha_0+\alpha_s+1=n_i$ for some
$i\in\{1,2,\ldots,t\}$, and there are $n_i$ possible choices for
the first and last coordinates of $\mathrm{type}(\mathbf{v})$. The
remaining coordinates $\alpha_1,\ldots,\alpha_{s-1}$ form a
permutation of the multiset
$\{n_1^{m_1},\ldots,n_i^{m_i-1},\ldots,n_t^{m_t}\}$. Hence, the
number of $(s-1)$-faces of $R_{k,q}$ whose interior vertices have
link isomorphic to $K_\beta$ equals
$$\sum_{i=1}^t n_i\frac{(s-1)!}{m_1! \cdots (m_i-1)! \cdots m_t!}
=\frac{k\cdot (s-1)!}{m_1!\cdots m_t!}.$$
\end{proof}
\noindent In the table below, we apply the above formula for
$k=6$.

$$\begin{tabular}{|c|c||c|c|c||c|c|c||c|c||c||c|}
  \hline
   $\bm{\lambda}\in \mathrm{Par}(6)$ & 6 & 51 & 42
  & 33 & 411 & 321 & 222 & 3111 & 2211 & 21111 & 111111 \\
  \hline
  \# of faces & 6 & 6 & 6 & 3 & 6 & 12 & 2 & 6 & 9 & 6 & 1 \\
  \hline
\end{tabular}$$

\begin{cor} Let $\bm{\beta}=(n^{m_1}_1,n^{m_2}_2,\ldots,n^{m_t}_t)$ be a
partition of $k$ into $s$ parts. The number of vertices of
$T_{k,q}$ whose link is isomorphic to $K_{\beta}$ equals
$$\frac{(q-1)! k}{(q-s)!m_1!\cdots m_t!}.$$
\end{cor}
%%%%%%%%%%%%%%%%%%%%%%%%%%%%%%%%%%%%%%%%%%%%%%%%%

\section{Links and stars of faces}

In this section we show that the link of any face of $T_{k,q}$ can
be expressed as a join of complexes associated with integer
partitions.
\begin{thm}\label{T:linkF}
The link of a $(t-1)$-dimensional face $F$ in $T_{k,q}$ is the
join of $t$ simplicial complexes, each of which is a reduced order
complex of a product of chains.
\end{thm}
\begin{proof}
Let $F=\conv\big\{\mathbf{v}^{(1)},\mathbf{v}^{(2)},
\ldots,\mathbf{v}^{(t)}\big\}$ be a $(t-1)$-dimensional face of
$T_{k,q}$. By (\ref{E:istasrana}), for all $i=1,2,\ldots,t-1$, we
have
$$\mathbf{v}^{(i+1)}=\mathbf{v}^{(i)}+\sum_{j\in S_i}e_j
\textrm{ for some }\emptyset\neq S_i\subset[k-1].$$  The sets
 $S_1,S_2,\ldots,S_{t-1}$ are nonempty, pairwise disjoint, and $S_i$ records
 the coordinates where $\mathbf{v}^{(i)}$ and $\mathbf{v}^{(i+1)}$
 differ. We let
 $S_t=[k]\setminus\big(S_1\cup S_2\cup \cdots\cup S_{t-1}\big)$.

Note that $\mathrm{link}_{T_{k,q}}(F)\subseteq
\mathrm{link}_{T_{k,q}}(\mathbf{v}^{(1)}) $, and we can identify the facets
of $\mathrm{link}_{T_{k,q}}(\mathbf{v}^{(1)})$ with a subset of
$\mathbb{S}_k$, see Theorem \ref{T:linkunutr} and Theorem
\ref{T:linkjeproduct}.
\begin{figure}[h]
\begin{center}
\begin{tikzpicture}[scale=1.06]
\draw (0,0) -- (2,2) -- (0,4) -- (2,6) --
(0,8)--(-1,7)--(2,4)--(-1,1)--(0,0); \draw (1,1) -- (0,2); \draw
(0,6) -- (1,7);\foreach \x in {0,...,2}{\draw [fill] (\x,\x)
circle [radius=0.051];}\foreach \x in {0,...,3}{\draw [fill]
(\x-1,\x+1) circle [radius=0.051];}

\foreach \x in {0,...,2}{\draw [fill] (\x,\x+4) circle
[radius=0.051];}\draw [fill] (0,6) circle [radius=0.051]; \draw
[fill] (1,7) circle [radius=0.051];\draw [fill] (-1,7) circle
[radius=0.051];\draw [fill] (0,8) circle [radius=0.051];

\node at (0,-.23) {$_{\mathbf{v}^{(1)}=(1,1,3,3,3,6,7)}$};\node at
(0,8.23) {$_{(1,1,3,3,3,6,7)}$}; \node at (2.1,1)
{$_{(1,2,3,3,3,6,7)}$}; \node at (3,2)
{$_{(2,2,3,3,3,6,7)}$};\node at (-2.1,1)
{$_{(1,1,3,3,3,6,8)}$};\node at (-1.1,2)
{$_{(1,2,3,3,3,6,8)}$};\node at (2.5,3)
{$_{\mathbf{v}^{(2)}=(2,2,3,3,3,6,8)}$};\node at (3,4)
{$_{(2,2,3,3,4,6,8)}$}; \node at (2.5,5)
{$_{\mathbf{v}^{(3)}=(2,2,3,3,4,7,8)}$};\node at (3,6)
{$_{(1,1,2,2,3,6,7)}$};\node at (2.12,7)
{$_{(1,1,2,3,3,6,7)}$};\node at (-2.1,7) {$_{(2,2,4,4,4,7,8)}$};
\node at (-1.1,6) {$_{(2,2,3,4,4,7,8)}$};\node at (-1.1,4)
{$_{(2,2,3,3,3,7,8)}$}; \node at (-6,1.9) {$S_1=\{1,2,7\}$}; \node
at (-6,4) {$S_2=\{5,6\}$};\node at (-6,7.1) {$S_3=\{3,4,8\}$};

\foreach \x in {0,...,1}{\node at (.6-\x,\x+.4) {$_2$};} \foreach
\x in {0,...,2}{\node at (\x-.17,\x+.5) {$_7$};} \node at
(1.6,1.4) {$_1$};\node at (.6,2.4) {$_1$};\node at (.6,4.4)
{$_5$};\node at (1.65,3.4) {$_5$};\node at (.4,3.4) {$_6$};\node
at (1.44,4.4) {$_6$};\node at (.45,5.4) {$_4$};\node at (1.45,6.4)
{$_4$};\node at (0.4,7.4) {$_3$};\node at (-0.6,6.4) {$_3$};

\node at (1.51,5.4) {$_8$};\node at (.335,6.22) {$_8$};\node at
(-.66,7.5) {$_8$};

\end{tikzpicture}
\caption{The poset $P_F$ for
$F=\{\mathbf{v}^{(1)},\mathbf{v}^{(2)},\mathbf{v}^{(3)}\}$ in
$T_{8,9}$}\label{fig:linkstrane}

\end{center}
\end{figure}
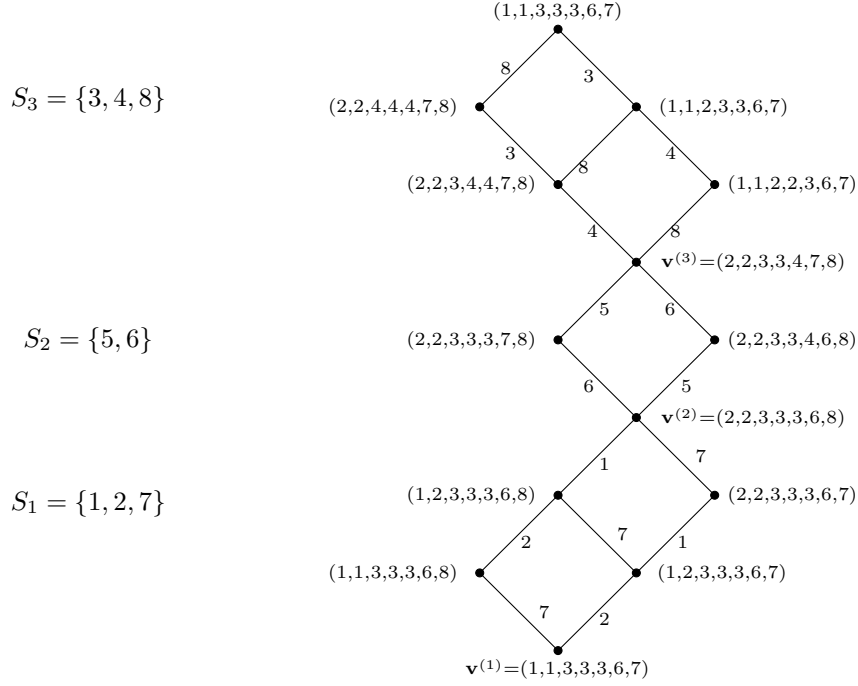

Let $\mathbb{S}_F $ denote the set of all permutations $\pi\in
\mathbb{S}_k$ such that the facet of
$\mathrm{link}_{T_{k,q}}(\mathbf{v}^{(1)})$ defined by $\pi$ also contains
all $\mathbf{v}^{(i)}$, for $i=2,3,\ldots,t$. These permutations
satisfy the following two conditions:
\begin{enumerate}
    \item[(i)] For each $i = 1, \dots, t-1$,
    the numbers in $S_i$ appear before $k$ in $\pi$,
   since the coordinates corresponding to these indices cannot decrease.

    \item[(ii)] The elements of each $S_i$ form a consecutive block in $\pi$,
     following $S_{i-1}$ and preceding $S_{i+1}$
     (recall that the order of elements before $k$ is
     reversed in $\pi'$, see (\ref{E:stranaulinku})).
     Consequently, all elements of $S_t$ appear at the end of $\pi$.
\end{enumerate}

\noindent In short, $\pi\in\mathbb{S}_F$ can be schematically
written as $\pi = S_1|S_2|\cdots|S_{t-1}|S_t$.

\noindent
 Furthermore, there may be some additional conditions on $\pi \in
\mathbb{S}_F$. If some of the coordinates of $\mathbf{v}^{(i)}$
indexed by elements from $S_i$ are mutually equal, or some of them
are equal $0$ or $q$, we have the situation as in the proof of
Theorem \ref{T:linkjeproduct} (the conditions (1)$-$(3)).

 For all $i=1,2,\ldots,t$, we apply
 the same reasoning as
in the proof of Theorem \ref{T:linkjeproduct} to define the poset
$P_{S_i}$ as the product of chains. For example, if $\{c, c+1,
\ldots, c+p\}\subset S_i$ are the coordinates of
$\mathbf{v}^{(i)}$ with equal values, they form a chain of length
$p+1$, which is a factor of $P_{S_i}$. We label the edges of this
chain by $c+p, \cdots, c+1, c$ (from minimum to maximum) and
transfer these labels to $P_{S_i}$. The poset $P_{S_i}$ is the
product of such chains whose total length is $|S_i|$. Finally, we
define $P_F$ as the poset obtained by identifying the minimum of
$P_{S_i}$ with the maximum of $P_{S_{i-1}}$ for all
$i=2,3,\ldots,t$. For $i>j$, all the elements of $P_{S_i}$ are
greater than all the elements of $P_{S_j}$ in $P_F$, see Figure
\ref{fig:linkstrane}. The permutations in $\mathbb{S}_F$, which
define the facets of $\mathrm{link}_{T_{k,q}}(F)$, are in bijection with
the labels of maximal chains of $P_F$. Thus,
$\mathrm{star}_{T_{k,q}}(F)\cong \Delta(P_F)$ and
\begin{equation}\label{E:linkoffaseisajoin} \mathrm{link}_{T_{k,q}}(F)\cong
\Delta(\overline{P}_{S_1})*\Delta(\overline{P}_{S_{2}})*\cdots*
\Delta(\overline{P}_{S_t}).
\end{equation}
In the same way as in the proof of Theorem \ref{T:linkjeproduct},
we conclude that the vertices of $\mathrm{star}_{T_{k,q}}(F)$ correspond to
the elements of $P_F$, as shown in Figure \ref{fig:linkstrane}.
 \end{proof}
\begin{rem}\label{R:linkasjoin}The sizes of the sets $S_i$ define a
partition $\bm{\lambda}=(\lambda_1,\ldots,\lambda_t)$ of $k$ into
$t$ parts. If $\lambda_i=|S_j|$, then the lengths of the chains in
the poset $P_{S_j}$
    determine the partition $\bm{\sigma}_i$ of $\lambda_i$.
 From
(\ref{E:linkoffaseisajoin}) it follows that
\begin{equation}\label{E:Joinofpartition} \mathrm{link}_{T_{k,q}}(F)\cong
K_{\sigma_1}*K_{\sigma_2}*\cdots*K_{\sigma_t}.
\end{equation}

\end{rem}
 \noindent From the
above discussion we are able to describe the links of the faces of
$T_{k,q}$.
\begin{thm}\label{T:linksoffaces} The link of a $(t-1)$-dimensional
 face of $T_{k,q}$ is uniquely determined by a pair $(\bm{\lambda}, M)$, where
$\bm{\lambda}=(\lambda_1,\ldots,\lambda_t)$ is a partition of $k$
into $t$ parts and
$M=\{\bm{\sigma}_1,\bm{\sigma}_2,\ldots,\bm{\sigma}_t\}$ is a
(multi)set containing, for each $i$, a partition $\bm{\sigma}_i$
of $\lambda_i$.
\end{thm}
 We observe the
following:
\begin{enumerate}
    \item[(1)] The total number of parts in all
$\bm{\sigma}_i$ is at most $\min\{q,k\}$.
    \item[(2)] The representation of $\mathrm{link}_{T_{k,q}}(F)$ given
    in (\ref{E:Joinofpartition}) is not always unique. As an illustration,
     a $1$-simplex
    appears as the link of a $1$-dimensional face in $T_{4,q}$, and it
    can be represented as
    $K_{(3)}*K_{(1)}\cong K_{(2)}*K_{(2)}\cong \Delta^1$. \item[(3)]
    If every partition $\bm \sigma_i$ has more than two
parts whenever $\lambda_i>1$, then the representation of
$\mathrm{link}_{T_{k,q}}(F)$ in (\ref{E:Joinofpartition}) is unique, by
Theorem \ref{T:razparrazkom} and induction.
\end{enumerate}
We now collect some consequences that follow directly from
Theorems \ref{T:razparrazkom} and \ref{T:linksoffaces}.
\begin{cor}\label{C:osobine} Assume that the link of a
$(t-1)$-dimensional face $F$ of $T_{k,q}$ is determined by the
pair $(\bm{\lambda},M)$ as in Theorem \ref{T:linksoffaces}.
\begin{itemize}
\item[($i$)] If $\bm
\sigma_i=(\sigma_{i,1},\sigma_{i,2}\ldots,\sigma_{i,{a_i}})\in
\mathrm{Par}(\lambda_i)$ for all $i\in [t]$, then the number of facets of
$\mathrm{link}_{T_{k,q}}(F)$ is
$$\frac{\lambda_1!}{\sigma_{1,1}!\sigma_{1,2}!\cdots\sigma_{1,a_1}!}
\cdot\frac{\lambda_2!}{\sigma_{2,1}!\sigma_{2,2}!\cdots\sigma_{2,a_2}!}\cdot
\cdots \cdot
\frac{\lambda_t!}{\sigma_{t,1}!\sigma_{t,2}!\cdots\sigma_{t,{a_t}}!}.
$$

\item[($ii$)] The parts of $\bm{\lambda}$ that are equal to $1$ do
not affect the link of $F$.
 \item[($iii$)] The link of a $(t-1)$-face of
$T_{k,q}$ is a $(k-t-1)$-dimensional complex.
    \item[($iv$)] If $q>k-t$, then
    there are $p_{k,t}$
    combinatorially distinct spheres among the links of all
    $(t-1)$-dimensional faces of $T_{k,q}$
    \item[($v$)] For any fixed $m\in \mathbb{N}$,
    all $m$-dimensional
    simplicial complexes that appear as links in
    $T_{k,q}$ (for all $k$ and $q$) already occur in $T_{2m+2,2m+2}$.
    \item[($vi$)] The number of $m$-dimensional spheres among
    the links of faces of all complexes $T_{k,q}$ (for all $k$ and $q$)
     is $p_{m+1}$.
\end{itemize}
\end{cor}
\begin{proof}The statement ($i$) follows from (\ref{E:Joinofpartition})
and (\ref{E:facets}), and ($ii$) follows from the fact that
$K_{(1)}=\emptyset$. To prove ($iii$), we calculate
$$\mathrm{dim }\, \mathrm{link}_{T_{k,q}}(F)=(\lambda_1-2)+
(\lambda_2-2)+\cdots+(\lambda_t-2)+t-1=k-t-1.$$
\begin{enumerate}
\item[($iv$)] From Theorem \ref{T:razparrazkom} and equation
(\ref{E:Joinofpartition}) it follows that $\mathrm{link}_{T_{k,q}}(F)$ is a
sphere if and only if $\bm \sigma_i=(1,1,\ldots,1)$ (with
$\lambda_i$ ones) for all $i$.

\item[($v$)] From ($iii$) it follows that an $m$-dimensional
simplicial complex is the link of a $(k-m-2)$-dimensional face of
$T_{k,q}$. If $k>2m+2$, then any partition of $k$ into $k-m-1$
summands has at least $k-2m-2$ summands equal to $1$, and they can
be ignored by ($ii$). \item[($vi$)] From (\ref{E:Joinofpartition})
it follows that the link of a face is an $m$-dimensional sphere if
and only if $\lambda_1+\lambda_2+\cdots+\lambda_t=m+t+1$ and the
corresponding partitions are $\bm{\sigma}_i=(1,1,\ldots,1)$ (with
at least two parts). In that case
$(\lambda_1-1,\lambda_2-1,\ldots,\lambda_t-1)$ is a partition of
$m+1$.\end{enumerate}\end{proof}

We now turn to enumerating the combinatorially distinct
$(m-1)$-dimensional complexes that appear as the links in some
$T_{k,q}$. Such a complex is determined by a partition $\bm
\lambda=(\lambda_1,\lambda_2,\ldots, \lambda_t)\textrm{, such that
} \lambda_1+\lambda_2+\ldots+ \lambda_t=m+t,$ and
$\{\bm{\sigma_1},\bm{\sigma_2},\ldots,\bm{\sigma_t}\}$ where $\bm
\sigma_i \in \mathrm{Par} (\lambda_i)$, see Theorem \ref{T:linksoffaces}.
If $\bm{\sigma}_i=(\lambda_i)$, then
    $K_{\sigma_i} \cong \Delta_{\lambda_i-2}$, and the join of these
    complexes in (\ref{E:Joinofpartition}) is $\Delta^{p-1}$,
    where
    $p=\sum_{i:\sigma_i=(\lambda_i)}\left(\lambda_i-1\right)$.
    An $(m-1)$-dimensional link in
$T_{k,q}$ has the form $\Delta^{p-1}*L$, where $L$ is an
$(m-p-1)$-dimensional complex obtained as the join of complexes
$K_{\sigma_i}$ corresponding to the partitions $\bm \sigma_i$ with
at least two parts. Although $\Delta^{p-1}$ admits many join
representations, the join representation of $L$ is unique (see (2)
and (3) of Remark \ref{R:linkasjoin}).
    Let $Q_s$ denote the number of $(s-1)$-dimensional complexes that can be
expressed as joins of complexes $K_\sigma$, where each partition
$\bm{\sigma}$ has at least two parts. We set $Q_0=1$, counting the
empty complex.
\begin{lem}\label{L:Q_s}For all $s\in \mathbb{N}$ we have
$$
Q_s=\sum_{\substack{\mu\in \mathrm{Par}(s)\\\mu=(n_1^{m_1},\ldots,
n_j^{m_j})} } \left({p_{n_{1}+1}-1\choose
m_1}\right)\cdots\left({p_{n_{j}+1}-1\choose m_j}\right)
$$
\end{lem}
\begin{proof}
Consider the complexes of the form
$K_{\sigma_1}*\cdots*K_{\sigma_v}$, where $\lambda_1
+\lambda_2+\cdots+\lambda_v=s+v$, and each $\bm{\sigma}_i$ is a
partition of $\lambda_i>1$ with at least two parts. Recall that
all complexes obtained in this way are distinct. For any such
partition $(\lambda_1,\lambda_2,\ldots,\lambda_v)$ of $v+s$, we
consider the partition $\bm{\mu}=(n_1^{m_1},\ldots,
n_j^{m_j})=(\lambda_1-1,\lambda_2-1,\ldots,\lambda_v-1)\in
\mathrm{Par}(s).$ For each $n_i$ we must choose $m_i$ from
$p_{n_i+1} -1$ possible partitions, excluding the one consisting
of a single part.
\end{proof}
Here are some values of $Q_s$.
$$\begin{tabular}{|c|c|c|c|c|c|c|c|c|c|c|c|}
  \hline
  $s$ & 0 & 1 & 2 & 3 & 4 & 5 & 6 & 7 & 8 & 9 & 10\\
  \hline
  $Q_s$ & 1 & 1 & 3 & 7 & 16 & 34 & 74 & 151 & 312 & 625 & 1245 \\
  \hline
\end{tabular}$$
The generating function for the sequence $(Q_s)$, obtained from
Lemma~\ref{L:Q_s}, is given by
\begin{equation}\label{E:genforQ_s}
   Q(x) = \sum_{s \in \mathbb{N}_0} Q_s x^s
   = \prod_{n \in \mathbb{N}} (1 - x^n)^{1 - p_{n+1}}.
\end{equation}
Finally, we determine $C_m$, the number of all $(m-1)$-dimensional
links in all $T_{k,q}$.

\begin{thm}
The generating function for the sequence $(C_n)$ is
$$C(x)=\sum_{m\in \mathbb{N}_0}C_mx^m=\frac{Q(x)}{1-x}=
\frac{1}{1-x}\prod_{n\in\mathbb{N}}(1-x^n)^{1-p_{n+1}}.$$
\end{thm}\begin{proof} The statement follows from the fact that
$C_m=Q_m+Q_{m-1}+\cdots+Q_1+Q_0$ and (\ref{E:genforQ_s}).
\end{proof}
\noindent The table below shows some values for the number of
distinct $(m-1)$-dimensional links in the complexes $T_{k,q}$:
$$\begin{tabular}{|c|c|c|c|c|c|c|c|c|c|c|c|}
  \hline
  $m$ & 0 & 1 & 2 & 3 & 4 & 5 & 6 & 7 & 8 & 9 & 10\\
  \hline
  \# of links & 1& 2 & 5 & 12 & 28 & 62 & 136 & 287 & 599 & 1224 & 2469\\
  \hline
\end{tabular}$$
 In the same way, with somewhat complicated calculations,
 we can obtain the number of combinatorially
different links of all $m$-dimensional faces in $T_{k,q}$.
%%%%%%%%%%%%%%%%%%%%%%%%%%%%%%%%%%%%%%%%%%%%%%%%%

\section{The star cluster of a facet and its $h$-vector}
 The notion of star cluster was introduced in
\cite{Barm}. The \textit{star cluster} of a face $\sigma$ in a
simplicial complex $K$ is $SC_K(\sigma)= \bigcup_{v\in
\sigma}\mathrm{star}_K(v)$. In this section we will analyze the star
cluster of a facet $F=F(\mathbf{v},Id)=
 \conv\left\{\mathbf{v} ^{(1)},\mathbf{v}
^{(2)}, \ldots,\mathbf{v}^{(k)}\right\}$ in $T_{k,q}$. For $1\leq
i<j\leq k$, equation \eqref{E:vertices1} yields
\begin{equation}\label{E:Razlika}
\mathbf{v}^{(j)} -\mathbf{v}^{(i)} = e_{k+1-j}+e_{k+2-j}+\cdots
+e_{k-i}=
(\underbrace{0,\ldots,0}_{k-j},\underbrace{1,\ldots,1}_{j-i},
\underbrace{0,\ldots,0}_{i-1}).
\end{equation}
\noindent Throughout this section, assume $
\mathbf{v}=\mathbf{v}^{(1)}=(v_1,v_2,\ldots,v_{k-1})$ satisfies
$0<v_1<v_2<\cdots<v_{k-1}<q-1$, so that all vertices of $F$ lie in
the interior of $R_{k,q}$.

For each $i$, let $\mathcal{F}_i$ denote the set of the facets of
$\mathrm{star}_{T_{k,q}}(\mathbf{v}^{(i)})$. As noted in Remark
\ref{R:star}, we identify $\mathcal{F}_i$ with $\mathbb{S}_k$.

\begin{prop}\label{P:presjek}
The facet of $\mathcal{F}_j$ indexed by $\pi \in \mathbb{S}_k$ is
also contained in $\mathcal{F}_i$ for some $i<j$ if and only if
the last $j-i$ entries of $\pi$ belong to the set $\{k-j+1, k-j+2,
\dots, k-i\}$.
\end{prop}

\begin{proof}
We use the bijection between $\mathbb{S}_k$ and $\mathcal{F}_j$
described in (\ref{E:stranaulinku}), Theorem \ref{T:linkunutr}. If
the facet $F(\mathbf{v}^{(j)}_{\pi}, \pi') \in \mathcal{F}_j$
determined by $\pi\in \mathbb{S}_k$ also contains
$\mathbf{v}^{(i)}$, then, by (\ref{E:Razlika}), the coordinates of
$\mathbf{v}^{(j)}_{\pi}$ at positions $k-j+i+1, k-j+i+2, \dots,
k-i$ are strictly smaller than those of $\mathbf{v}^{(j)}$.
Therefore, $k-j+i+1, k-j+i+2, \dots, k-i$ appear in $\pi$ after
$k$ (cf. (\ref{E:stranaulinku})). Moreover, after increasing these
coordinates of $\mathbf{v}^{(j)}_{\pi}$ at these positions as in
(\ref{E:vertices1}), we must recover $\mathbf{v}^{(j)}$. It then
follows from (\ref{E:stranaulinku}) that the numbers $k-j+i+1,
k-j+2, \dots, k-i$ appear at the end of $\pi$.
\end{proof}

Let $\mathcal{F}=\mathcal{F}_1\cup\cdots\cup\mathcal{F}_k$ denote
the set of all facets of $SC_{T_{k,q}}(F)$.
\begin{thm} The number of facets in $SC_{T_{k,q}}(F)$ is
 \begin{equation}\label{E:FUI}
|\mathcal{F}|=k\cdot k!+\sum_{t=2}^k(-1)^{t-1}\sum_{1\leq
i_1<\cdots<i_t\leq k}(i_2-i_1)!(i_3-i_2)!\cdots(k-i_t+i_1)!.
\end{equation}
\end{thm}
\begin{proof} By Proposition \ref{P:presjek} we have
$$|\mathcal{F}_{i_1}\cap \mathcal{F}_{i_2}\cap\cdots\cap \mathcal{F}_{i_t}|=
(i_2-i_1)!(i_3-i_2)!\cdots(k-i_t+i_1)!.$$ Applying the principle
of inclusion-exclusion now yields~\eqref{E:FUI}.
\end{proof}
Another formula for the number of facets in $SC_{T_{k,q}}(F)$ can
be obtained by combining Proposition \ref{P:faceswithprescribed
partition} and ($i$) of Corollary \ref{C:osobine}.
\begin{equation}\label{E:FIPconcrete}
|\mathcal{F}|=\sum_{s=1}^{k}(-1)^{s-1}k\cdot(s-1)!
\sum_{\substack{\lambda=(\lambda_1,\ldots,\lambda_s)\in \mathrm{Par}(k)
\\\lambda=(n_1^{m_1},\ldots,n_j^{m_j})}}
\frac{\lambda_1!\cdots\lambda_s!}{m_1!\ldots m_j!}.
\end{equation}
The facets of $SC_{T_{k,q}}(F)$ also can be described as a subset
of $\mathbb{S}_{k+1}$ using a new (up to our knowledge)
permutation statistic.
\begin{defn}
 For a
permutation $\pi=\pi_1\pi_2\ldots\pi_n \in \mathbb{S}_n$ we define
the \textit{faithful initial part} (or simply \textit{init}) of
$\pi$ as $\mathrm{init}(\pi)=
\min\big\{t:\big\{\pi_1,\pi_2,\ldots,\pi_t\big\}=[t]\big\}.$ Let
$X_n$ denote the number of permutations $\pi \in \mathbb{S}_n$ for
which $\mathrm{init}(\pi)=n$.
\end{defn}
\begin{prop}\label{P:rekrelzaX_n}
The numbers $X_n$ satisfy the following recursive relation:
\begin{equation}\label{E:relzaX_n}
    X_n=n!-(n-1)!X_1-(n-2)!X_2-\cdots-2!X_{n-2}-X_{n-1}.
\end{equation}
\end{prop}
\begin{proof}
Note that the first $j$ entries of a permutation of $\mathbb{S}_n$
whose faithful initial part is $j$ form a permutation of
$\mathbb{S}_j$ whose init is $j$. Therefore, for $j<n$, there are
exactly $X_j(n-j)!$ permutations from $\mathbb{S}_n$ whose init is
exactly $j$.
\end{proof}
\begin{thm}\label{T:facetsareinit} If all of the vertices of the facet
$F=F(\mathbf{v},Id)$ are in the interior of $R_{k,q}$, the number
of facets of $SC_{T_{k,q}}(F)$ is $X_{k+1}$.
\end{thm}
\begin{proof}
We use the following decomposition of
$\mathcal{F}$\begin{equation}\label{E:disjointunion}
\mathcal{F}=\mathcal{F}_1 \sqcup \Big(\mathcal{F}_2\setminus
\mathcal{F}_1\Big)\sqcup\cdots \sqcup\Big(\mathcal{F}_j
\setminus\bigcup_{i=1}^{j-1}\mathcal{F}_i\Big)\sqcup\cdots
\sqcup\Big(\mathcal{F}_k\setminus\bigcup_{i=1}^{k-1}\mathcal{F}_i\Big).
\end{equation}
For each $j\in[k]$, we describe a bijection between the
$\mathcal{F}_j\setminus(\mathcal{F}_1\cup
   \cdots\cup \mathcal{F}_{j-1})$ and
   $\{\pi\in \mathbb{S}_{k+1}:\mathrm{init}(\pi)=k+1\textrm{, } \pi_j=k+1\}$.
   Recall that a permutation
   $\pi=\pi_1\pi_2\ldots\pi_k \in \mathbb{S}_k$ that defines a facet of
$\mathcal{F}_j\setminus(\mathcal{F}_1\cup
   \cdots\cup \mathcal{F}_{j-1})$ satisfies the
   following $j-1$ conditions, see Proposition \ref{P:presjek}:
   $$\{\pi_{k-j+1+i},\ldots,\pi_{k}\}
   \neq\{k-j+1,\ldots,k-i\}\textrm{ for all }i=1,2,\ldots,j-1.$$
\noindent We map this facet (the corresponding permutation $\pi$)
to the permutation $\widehat{\pi}\in \mathbb{S}_{k+1}$ by
reversing the order, relabeling $t\mapsto t+j \textrm{ (mod } k)$
and inserting $k+1$ at the position $j$
$$\widehat{\pi}=(\pi_k+j)\cdots (\pi_{k-j+2}+j) k+1
(\pi_{k-j+1}+j)\cdots (\pi_{1}+j) \textrm{ }(\textrm{the sum is
modulo } k).$$ This ensures that $\mathrm{init}(\widehat{\pi})=k+1$, with
$k+1$ placed exactly at the position $j$.
\end{proof}
Observe that (\ref{E:FUI}) and (\ref{E:FIPconcrete}) provide exact
formulas for $X_{k+1}$.

 We exploit a slightly modified
lexicographic order $<_L$ on $\mathbb{S}_k$ (a shelling order for
 $Sd(\partial\Delta^{k-1})$, described in (\ref{E:lexshell})) to
obtain an explicit shelling of $SC_{T_{k,q}}(F)$.
\begin{prop}\label{P:myshel} The linear order
$<_I$ on $\mathbb{S}_k$ defined by
\begin{equation}\label{E:newshelling}
\pi <_I \sigma \Longleftrightarrow  \left\{%
\begin{array}{lc}
    \mathrm{init}(\pi)<\mathrm{init}(\sigma), & \hbox{\textrm{or};} \\
    \pi<_L \sigma, &
    \hbox{\textrm{if }} \mathrm{init}(\pi)=\mathrm{init}(\sigma)\\
\end{array}%
\right.
\end{equation}
 is a shelling order for $Sd(\partial\Delta^{k-1})$.
\end{prop}
\begin{proof}
Recall that $\pi=\pi_1\pi_2\ldots \pi_k\in \mathbb{S}_k$ defines
the facet $F_\pi=\big\{\{\pi_1\},\{\pi_1,\pi_2\},
\ldots,[k]\setminus \{\pi_k\} \big\}$ of
$Sd(\partial\Delta^{k-1})$. We distinguish two cases.

 $1^\circ$
$\mathrm{init}(\pi)=a<b=\mathrm{init}(\sigma)$. In that case, we have that
$b=\sigma_i$ for some $i<b$ (otherwise $\mathrm{init}(\sigma)\neq b$), and
$\sigma_i=b>\sigma_{i+1}$. The permutation
$\sigma'=\sigma_1\ldots\sigma_{i-1}\sigma_{i+1}\,
b\,\,\sigma_{i+2} \ldots\sigma_k$ satisfies $\sigma'<_I\sigma$,
$F_\pi\cap F_\sigma\subset F_{\sigma'}\cap F_\sigma=
F_{\sigma}\setminus\{\sigma_1,\ldots,\sigma_i\}$, so
condition~\eqref{E:defshell} holds.

$2^\circ$ $\mathrm{init}(\pi) =\mathrm{init}(\sigma)$ and $\pi<_L \sigma$. We
consider (as in the lexicographical shelling) the descent set
$D(\sigma)=\{j:\sigma_j>\sigma_{j+1}\}$ and choose the minimal
$i\in D(\sigma)$ such that $\{\pi_1,\pi_2,\ldots,\pi_i\}\neq
\{\sigma_1,\sigma_2,\ldots,\sigma_i\}.$ Such $i$ has to exist
because $\pi<_L\sigma$. Then, we define
$\sigma'=\sigma_1\ldots\sigma_{i-1}\sigma_{i+1}\,\sigma_{i}\sigma_{i+2}
\ldots\sigma_k$, and obtain that $F_\pi\cap F_\sigma\subset
F_{\sigma'}\cap F_\sigma=
F_{\sigma}\setminus\{\sigma_1,\ldots,\sigma_i\}$.
\end{proof}
Although the shelling order $<_I$ differs from the standard
lexicographic order $<_L$ defined in Example 1, the type of each
facet remains the same in both orders. In fact, for every $\pi \in
\mathbb{S}_k$, we have $\mathrm{type}(F_\pi)=\mathrm{des}(\pi)$ with respect to
$<_I$ as well as to $<_L$. This is because the restriction of
$F_\pi$ depends only on the descent positions of $\pi$, which are
unaffected by the tie-breaking rule used in $<_I$.

 Let $\mathrm{Init}_k(t)$ denote the subcomplex of
$Sd(\partial\Delta^{k-1})$ spanned by the facets corresponding to
the permutations of $\mathbb{S}_k$ whose init is $t$. The
advantage of the shelling order $<_I$ (for our purposes) is that
it lists all of the facets from $\mathrm{Init}_k(1)$ first, then those from
$\mathrm{Init}_k(2)$, and so on.

We define a $k\times k$ matrix $H_k=\big [h^{k}_{i,d}\big]_{i,d\in
[k]}$, where each entry counts the number of permutations in
$\mathbb{S}_k$ with a prescribed number of descents and a given
initial part: $$ h^{k}_{i,d}=\big|\{\pi\in
\mathbb{S}_k:\mathrm{init}(\pi)=i, \mathrm{des}(\pi)=d-1\}\big| .$$

Thus, the numbers $h^k_{i,d}$ provide a joint refinement of the
two permutation statistics: the number of descents and the initial
part.

 The
column sums of $H_k$ are Eulerian numbers,
$h^k_{1,j}+h^k_{2,j}+\cdots+h^k_{k,j}=A(k,j-1)=
h_{j-1}(Sd(\partial \Delta^{k-1}))$. Let
$\mathbf{h}^k_t=(h^k_{t,1},\ldots,h^k_{t,k})\in \mathbb{N}_0^{k}$
  denote the $t$-th row of
 $H_k$. Note that
$$h^k_{t,1}+h^k_{t,2}+\cdots+h^k_{t,k}=
\big|\{\pi\in \mathbb{S}_{k}:\mathrm{init}(\pi)=t\}\big|=X_t(k-t)!.$$ We
can interpret $\mathbf{h}^k_t$ as the change in the
$\mathbf{h}$-vector of $Sd(\partial\Delta^{k-1})$ after adding all
the facets in the shelling order $<_I$ corresponding to
permutations whose initial part is $t$:
$$\mathbf{h}^k_t=\mathbf{h}\Big(\bigcup_{i=1}^t \mathrm{Init}(i)\Big)-
\mathbf{h}\Big(\bigcup_{i=1}^{t-1} \mathrm{Init}(i)\Big).$$ We also
describe the recursive relations for $\mathbf{h}^k_t$. Recall that
the convolution of two sequences is defined by
$(a_0,a_1,\ldots,a_s)*(b_0,b_1,\ldots,b_t)=
(c_0,c_1,\ldots,c_{s+t}),$ where
 $c_j=a_0b_j+a_1b_{j-1}+\cdots+a_jb_0.$
\begin{prop}
The sequences $\mathbf{h}^k_t$ satisfy the following recursive
relations:
\begin{itemize}
    \item[(i)] $\mathbf{h}^k_1=(\mathbf{h}(Sd(\partial\Delta^{k-2})),0);$
    \item[(ii)] For $1<t<k$ we have
    $\mathbf{h}^k_t=\mathbf{h}_t^t*\mathbf{h}(Sd(\partial\Delta^{k-t-1}));$
    \item[(iii)]$\mathbf{h}^k_k=
    \mathbf{h}(Sd(\partial\Delta^{k-1}))-\sum_{i=1}^{k-1}
 \mathbf{h}^k_i.$
\end{itemize}
\end{prop}
\begin{proof}
Statements $(i)$ and $(iii)$ follow directly from the definitions.
To prove $(ii)$, we have to count the number of the descents in
the permutations from $\mathbb{S}_k$ whose initial part is exactly
$t$. The first $t$ elements of such a permutation form a
permutation from $\mathbb{S}_t$, and their descents are counted by
the entries of $\mathbf{h}_t^t$. The remaining $k-t$ elements form
a permutation of $[k]\setminus [t]$, which has the same descent
distribution as the corresponding permutation in
$\mathbb{S}_{k-t}$. Convolution then combines these two
contributions, yielding the stated recursion.
\end{proof}
\begin{thm}\label{T:h-cluster}
If all vertices of $F=F(\mathbf{v},Id)$ lie in the interior of
$R_{k,q}$, then the entries of $\mathbf{h}$-vector of
$SC_{T_{k,q}}(F)$ are given by
$$
h_j(SC_{T_{k,q}}(F))=\sum_{\substack{\pi\in
\mathbb{S}_k\\\mathrm{des}(\pi)=j}}\mathrm{init}(\pi)=\sum_{i=1}^{k}ih^{k}_{i,j} .$$

\end{thm}In other words
$$\mathbf{h}(SC_{T_{k,q}}(F))=(1,2,\ldots,k)\cdot H_k=
\sum_{t=1}^k t\cdot\mathbf{h}^k_t=k\cdot
 \mathbf{h}(Sd(\partial\Delta^{k-1}))-\sum_{t=1}^{k-1}
(k-t)\cdot\mathbf{h}^k_t.$$ This formula expresses the
$\mathbf{h}$-vector of the star cluster as a weighted sum of the
contributions from each initial part, with later terms adjusted to
account for overlaps in the shelling order.
\begin{proof}
We will describe a concrete shelling of $SC_{T_{k,q}}(F)$ using
the decomposition of $\mathcal{F}$ described in
(\ref{E:disjointunion}), along with the shelling order $<_I$ of
$Sd(\partial\Delta^{k-1}) $ defined by (\ref{E:newshelling}).
Recall that the facets of $\mathcal{F}_i$ correspond to elements
of $\mathbb{S}_k$. First, we list all the facets from
$\mathcal{F}_1$ as the permutations from $\mathbb{S}_k$, ordered
according to $<_I$.
   For all $j=2,3,\ldots,k$, we continue the
    shelling of $SC_{T_{k,q}}(F)$ with the
    facets from $\mathcal{F}_j$ that not
    appear earlier in some $\mathcal{F}_i$ for $i<j$.
   Note that the map
   $$\pi_1\pi_2\ldots\pi_k \mapsto (\pi_k+j) (\pi_{k-1}+j)
   \ldots (\pi_1+j)\hspace{.6cm}\textrm{  (the sum is modulo $k$)}$$
   is a bijection between the facets from
   $\mathcal{F}_j\cap (\mathcal{F}_1\cup
    \mathcal{F}_2\cup \cdots \cup \mathcal{F}_{j-1})$ and
    permutations of $\mathbb{S}_k$ whose initial part
    is at most $j-1$  (here we identify
    $\mathcal{F}_j$ with $\mathbb{S}_k$). These
facets form the initial part of the shelling of
$Sd(\partial\Delta^{k-1})$ described in Proposition
\ref{P:myshel}, and they have already been listed earlier. We
continue the shelling of $SC_{T_{k,q}}(F)$ by successively adding
the new facets from
    $\mathcal{F}_j\setminus \big(
    \mathcal{F}_1\cup \cdots \cup \mathcal{F}_{j-1}\big)$
    following the order $<_I$, with $\mathcal{F}_j$ identified with $\mathbb{S}_k$.
Finally, the type of the facet defined by $\pi$ is $\mathrm{des}
(\pi)$, and the facets determined by $\pi$ appear exactly
$\mathrm{init} (\pi)$ times across all $\mathcal{F}_i$, $i\in
[k]$.
\end{proof}
Therefore, the entries of $\mathbf{h}\big(SC_{T_{k,q}}(F)\big)$
count descents in permutations from $\mathbb{S}_k$, with each
permutation contributing with multiplicity equal to its initial
part $\mathrm{init}(\pi)$. Although the facets of
$SC_{T_{k,q}}(F)$ (see Theorem~\ref{T:facetsareinit}) can be
identified with those of $\mathrm{Init}(k+1) \subset
Sd(\partial\Delta^{k})$, the two complexes have different
$\mathbf{h}$-vectors.

\section{Shelling of $T_{k,q}$}\label{S:shell}
We describe a concrete shelling of $T_{k,q}$ using the notation
explained in Proposition \ref{P:facetsarestrings}. Recall that
$\mathbf{a}=(a_1,a_2,\ldots,a_{k-1}) \in \mathcal{S}_{k,q}$ define
the facet
$F(\mathbf{a})=\big\{\mathbf{v}_a^{(1)},\mathbf{v}_a^{(2)},
\ldots,\mathbf{v}_a^{(k)}\big\}$, as described in
(\ref{E:vertices}). For a fixed $\mathbf{a}\in \mathcal{S}_{k,q}$,
there are at most $k$ facets of $T_{k,q}$ that share a common
ridge with $F(\mathbf{a})$, and we list all of these
possibilities, see Remark \ref{R:star}.
 \begin{equation}\label{E:izbacizadnjetjeme}
 \textrm{If }a_1>0,\textrm{ then }
F(\mathbf{a})\setminus \big\{\mathbf{v}_a^{(k)}\big\}=
F(\mathbf{a})\cap F(\mathbf{b}),\textrm{ for
}\mathbf{b}=(a_2,\ldots,a_{k-1},a_1-1).
\end{equation}
 \begin{equation}\label{E:izbaciprvotjeme} \textrm{If }
  a_{k-1}<q,\textrm{ then }
F(\mathbf{a})\setminus \big\{\mathbf{v}_a^{(1)}\big\}=
F(\mathbf{a})\cap F(\overline{\mathbf{b}}),\textrm{ for
}\overline{\mathbf{b}}=(a_{k-1}+1,a_1,\ldots,a_{k-2}).
\end{equation}

If $a_{i}\neq
     a_{i+1}$, for some $i\in[k-2]$, then
 \begin{equation}\label{E:izbacisrednje}
 F(\mathbf{a})\setminus \big\{\mathbf{v}_a^{(k-i)}\big\}=
  F(\mathbf{a})\cap
F(\mathbf{b}^{(i)}), \textrm{ where }
\mathbf{b}^{(i)}=(a_1,\ldots,a_{i+1},a_{i},\ldots,a_{k-1}).
\end{equation}

 For $\mathbf{a}=(a_1,a_2,\ldots,a_{k-1})\in
\mathcal{S}_{k,q}$, we define
 the\textit{ maximal coordinate} of $\mathbf{a}$ as $m(\mathbf{a})=
    max\{a_1,\ldots,a_{k-1}\}$ and
  the\textit{ total sum} of the coordinates
    $S(\mathbf{a})=\sum_{i=1}^{k-1}a_i$.

 We use the above parameters to define
 a linear order on the facets of $T_{k,q}$
$$F(\mathbf{a}) <F(\mathbf{b})\Longleftrightarrow \left\{%
\begin{array}{ll}
    m(\mathbf{a})<m(\mathbf{b}), & \hbox{or;} \\
    m(\mathbf{a})=m(\mathbf{b})\textrm{ and }
S(\mathbf{a}) <S(\mathbf{b}), & \hbox{or;} \\
   m(\mathbf{a})=m(\mathbf{b}), S(\mathbf{a})=S(\mathbf{b}), a_i>b_i,
    a_j=b_j\textrm{ for all }j<i.
   & \hbox{} \\
\end{array}%
\right.    $$
\begin{thm}
The above defined linear order $<$ is a shelling order for
$T_{k,q}$.

\end{thm}
\begin{proof}
 Assume that $F(\mathbf{a})$ precedes $F(\mathbf{b})$ in our
order. We now consider the following three cases and verify that
(\ref{E:defshell}) is satisfied.

$1^\circ$ Assume that $m(\mathbf{a})<m(\mathbf{b})=m$. In that
case, we are looking for the largest $i>1$ such that
$b_i=m>b_{i-1}$. If such $i$ exists, we swap $b_{i-1}$ and $b_i$,
and let
$\overline{\mathbf{b}}=(b_1,\ldots,m,b_{i-1},\ldots,b_{k-1})$.
Note that $F(\overline{\mathbf{b}})<F(\mathbf{b})$, and the vertex
$\mathbf{v}_b^{(k+1-i)}$ of $F(\mathbf{b})$ is not contained in
$F(\mathbf{a})$, because its $(k+1-i)$-th coordinate is $m+1$.
From (\ref{E:izbacisrednje}) we obtain that $F(\mathbf{a})\cap
F(\mathbf{b})\subseteq F(\mathbf{b})\cap
F(\overline{\mathbf{b}})=F(\mathbf{b})\setminus
\{\mathbf{v}_b^{(k+1-i)}\}.$
 If there is no such $i$, then $b_1=m$ and we let
$\overline{\mathbf{b}}=(b_2,b_3,\ldots,b_{k-1},m-1)$. Note that
$F(\overline{\mathbf{b}})<F(\mathbf{b})$, and from
(\ref{E:izbacizadnjetjeme}) it follows that $F(\mathbf{a})\cap
F(\mathbf{b})\subseteq F(\mathbf{b})\cap
F(\overline{\mathbf{b}})=F(\mathbf{b})
\setminus\{\mathbf{v}_b^{(k)}\}.$

 $2^\circ$
Assume that $m(\mathbf{a})=m(\mathbf{b})$ and
$S(\mathbf{a})<S(\mathbf{b})$. In that case, we know that the last
vertex $\mathbf{v}_b^{(k)}$ of $F(\mathbf{b})$ is not contained in
$F(\mathbf{a})$. If $b_1>0$, for
$\overline{\mathbf{b}}=(b_2,b_3,\ldots,b_{k-1},b_1-1)$ we have
that $F(\overline{\mathbf{b}})<F(\mathbf{b})$, and from
(\ref{E:izbacizadnjetjeme}) follows that $F(\mathbf{a})\cap
F(\mathbf{b})\subseteq F(\mathbf{b})\cap
F(\overline{\mathbf{b}})=F(\mathbf{b})\setminus
\{\mathbf{v}_b^{(k)}\}$. If $b_1=b_2=\cdots=b_{i-1}=0$, $b_i>0$,
then the vertex $\mathbf{v}_b^{(k+1-i)}$ of $F(\mathbf{b})$ is not
contained in $F(\mathbf{a})$. Indeed, if
$\mathbf{v}_b^{(k+1-i)}\in F(\mathbf{a})$, all of its $(i-1)$
first coordinates (they are equal to $0$) have to increase.
Therefore, we obtain that $S(\mathbf{a})\geq S(\mathbf{b})$, which
is a contradiction. In that case, for
$\overline{\mathbf{b}}=(\underbrace{0,\ldots,0,}
_{(i-2)\textrm{-zero's}}b_{i},0\ldots,b_{k-1})$ we have that
\mbox{$F(\overline{\mathbf{b}})<F(\mathbf{b})$,} and from
(\ref{E:izbacisrednje}) follows $F(\mathbf{a})\cap
F(\mathbf{b})\subseteq F(\mathbf{b})\cap
F(\overline{\mathbf{b}})=F(\mathbf{b})
\setminus\{\mathbf{v}_b^{(k+1-i)}\}$.

$3^\circ$ Assume that $m(\mathbf{a})=m(\mathbf{b})$,
$S(\mathbf{a})=S(\mathbf{b})$, $a_j=b_j$ for all $j<i$ and
$a_i>b_i$. If the first vertices of $F(\mathbf{a})$ and
$F(\mathbf{b})$ are equal, i.e.,
$\mathbf{v}^{(1)}_a=\mathbf{v}^{(1)}_b$, then
$\{b_{i+1},\ldots,b_{k-1}\}=
\left(\{a_{i+1},\ldots,a_{k-1}\}\cup\{a_i\}\right)
\setminus\{b_i\}.$ In that case, we are looking for the smallest
$s\geq i$ such that $b_s<b_{s+1}$. It has to exist because
$a_i>b_i$, and $a_i$ appears after $b_i$ in $\mathbf{b}$. Also,
$a_i$ appears in $\mathbf{b}$ after $b_s$, and all elements
between $b_i$ and $b_s$ are smaller than $a_i$. Then we let
$\overline{\mathbf{b}}=(b_1,\ldots,b_{s+1},b_s, \ldots,b_{k-1})$,
for which $F(\overline{\mathbf{b}})<F(\mathbf{b}) \textrm{  and
}F(\mathbf{a})\cap F(\mathbf{b})\subseteq F(\mathbf{b})\cap
F(\overline{\mathbf{b}})=F(\mathbf{b})\setminus
\{\mathbf{v}_b^{(k+1-s)}\}$ hold.

\noindent If $\mathbf{v}^{(1)}_a\neq \mathbf{v}^{(1)}_b$ (the
first and the last vertices of $F(\mathbf{a})$
 and $F(\mathbf{b})$ are different), and if $b_1>0$, we let
 $\overline{\mathbf{b}}=(b_2,b_3,\ldots,b_{k-1},b_1-1)$.
 In that case, we
 have that
$$  F(\overline{\mathbf{b}})<F(\mathbf{b})\textrm{  and }
 F(\mathbf{a})\cap F(\mathbf{b})\subseteq F(\mathbf{b})\cap
F(\overline{\mathbf{b}})=F(\mathbf{b})
\setminus\{\mathbf{v}_b^{(k)}\}.$$ If $b_1=b_2=\cdots=b_{i-1}=0$,
$b_i>0$ then we let
$\overline{\mathbf{b}}=(\underbrace{0,\ldots,0,}
_{(i-2)-\textrm{zero's}}b_{i},0\ldots,b_{k-1})$. In that case we
have $F(\overline{\mathbf{b}})<F(\mathbf{b})\textrm{ and }
F(\mathbf{a})\cap F(\mathbf{b})\subseteq F(\mathbf{b})\cap
F(\overline{\mathbf{b}})=F(\mathbf{b})
\setminus\{\mathbf{v}_b^{(k+1-i)}\}.$ The argument why
$\mathbf{v}_b^{(k+1-i)}\notin F(\mathbf{a})$ is the same as in the
previous case.

\end{proof}
 From the above theorem and relations
(\ref{E:izbacizadnjetjeme})$-$(\ref{E:izbacisrednje}), we conclude
that the restriction of the facet $F(\mathbf{a})$ in the shelling
order $<$ is $\mathcal{R}(F(\mathbf{a}))=
\big\{\mathbf{v}^{(k+1-i)}_a:a_{i-1}< a_{i}\big\}$, where we
assume that $a_{0}=0$. Therefore, we obtain a nice combinatorial
interpretation for the entries of
    $h$-vector of $T_{k,q}$.
\begin{cor}\label{C:comb}
The entries of $\mathbf{h}(T_{k,q})=(h_0,h_1,\ldots,h_{k-1})$
enumerate the sequences $(0,a_1, \allowbreak a_2,\ldots,a_{k-1})$
in $\{0\}\times \mathcal{S}_{k,q}$ with exactly $i$ strict
ascents. In other words,
\begin{equation}\label{E:h-vectcomb}
h_i = \Bigl|\bigl\{ (0,a_1,\ldots,a_{k-1}) \in \{0\} \times
\mathcal{S}_{k,q} : |\{ j : a_{j-1} < a_j \}| = i \bigr\}\Bigr|.
\end{equation}
\end{cor}
In sharp contrast, while the $h$-vector of the barycentric
subdivision of $\Delta^{k-1}$ counts the descents in all
permutations from $\mathbb{S}_k$, the $h$-vector of the edgewise
triangulation $T_{k,q}$ counts the strict ascents in all sequences
from $\{0\}\times \mathcal{S}_{k,q}$.

\noindent Some values of $h$-vector are easy to obtain by a simple
combinatorial reasoning. For example, we have that
$$h_1(T_{k,q})={k+q-1\choose k-1}-1,
h_{k-1}(T_{k,q})={q-1\choose k-1},h_i(T_{k,2})={k\choose 2i}.$$
\noindent Now, we use Corollary \ref{C:comb} to determine the
exact values for all $h_{i}(T_{k,q})$.
 Let $a_{j,e}^{(t)}$ denote the number of
sequences $(0,a_1,\ldots,a_{t-2},j)$, $a_i\in\{0,1,\ldots,q-1\}$
that have exactly $e$ ascents. Obviously, from
(\ref{E:h-vectcomb}) we have that
\begin{equation}\label{E:hviaa}
h_i(T_{k,q})=\sum_{j=0}^{q-1}a^{(k)}_{j,i}.
\end{equation} The numbers
$a_{j,e}^{(t)}$ satisfy the following recursive relation
\begin{equation}\label{E:recrel}
a_{j,e}^{(t)}=\sum_{p=0}^{j-1}a_{p,e-1}^{(t-1)}+
\sum_{p=j}^{q-1}a_{p,e}^{(t-1)}.
\end{equation}

Note that the numbers $a_{j,e}^{(t)}$ can be recognized as the
coefficients of the polynomial
$P_{t-1,q}(x)=\big(1+x+x^2+\cdots+x^{q-1}\big)^{t-1}$. By
considering
\begin{equation}\label{E;p}
P_{t,q}(x)=P_{t-1,q}(x)(1+x+x^2+\cdots+x^{q-1}),
\end{equation}
 we conclude that the coefficients
of the polynomials $\big(P_{t,q}(x)\big)_{t\in \mathbb{N}}$ also
satisfy the recursive relations (\ref{E:recrel}). Therefore, we
obtain that
\begin{equation}\label{:E:Polynomial}
P_{t,q}(x)=a^{(t+1)}_{0,0}+\sum_{e=1}
^{\left\lfloor\frac{t(q-1)}{q}\right\rfloor} \sum_{j=0}
^{q-1}a^{(t+1)}_{j,e}x^{eq-j}.
\end{equation}

\noindent From (\ref{E:hviaa}) and (\ref{:E:Polynomial}), we
conclude that $h_i$ is the sum of the coefficients of
$x^{(i-1)q+1}$, $x^{(i-1)q+2}$, $\ldots,x^{iq}$ in $P_{k-1,q}(x)$.
A similar result is obtained in \cite{Ata} (see Theorem 1.1)
algebraically.
 By (\ref{E;p}), the above sum can be
recognized as the coefficient of $x^{iq}$ in $P_{k,q}(x)$, and
after evaluating
$\big(1+x+\cdots+x^{q-1}\big)^k=\frac{(1-x^q)^k}{(1-x)^k}$, we
obtain the exact values for the entries of $h$-vector of
$T_{k,q}$.
\begin{thm} The entries of $h$-vector of
triangulation $T_{k,q}$ are given by the formula
\begin{equation}\label{E:hvector}
h_i(T_{k,q})=\sum_{j=0}^i(-1)^j {k\choose j}{(i-j)q+k-1\choose
k-1}.
\end{equation}
\end{thm}
 For $k=q$, the edgewise triangulation $T_{k,q}$ can be
recognized as the regular and unimodular triangulation of the
Newton polytope $\textrm{Newt}(s_\lambda(\mathbf{x}))$ for
$\lambda=(k)$, see \cite{Pay} and \cite{Bayer} for the details.
Our formula (\ref{E:hvector}) gives $h^*$-vector of
$\textrm{Newt}(s_\lambda(\mathbf{x}))$, as proven in Proposition
53 of \cite{Bayer}.

%%%%%%%%%%%%%%%%%%%%%%%%%%%%%%%%%%%%%%%%%%%%%%%%%

%BIBLIOGRAPHY
% You do not have to use the same format for your references, but
%    include everything in this file.
% If you use BibTeX to create a bibliography, copy the .bbl file into here.
% We recommend you use \doi{...} and \arxiv{...} like the examples below,
% as they give a short display form with an active link to the full url.

\end{document}